\documentclass{article}%
\usepackage{amsfonts}
\usepackage{amsmath}
\usepackage{amsthm}
\usepackage{amssymb}
\usepackage{graphicx}
\usepackage[utf8]{inputenc}
\usepackage{mathrsfs}
\usepackage{bbm}
\usepackage{mathtools}
\usepackage{datetime}

\usepackage{verbatim}
\usepackage{enumerate}
\usepackage{todonotes}
\usepackage{dirtytalk}
\usepackage{tikz-cd}
\usepackage{standalone}
\usepackage{fullpage}

\usepackage{MnSymbol}%
\providecommand{\U}[1]{\protect\rule{.1in}{.1in}}

\usepackage{aliascnt}

\newcommand{\mynewtheorem}[2]{
	\newaliascnt{#1}{dummy}
	\newtheorem{#1}[#1]{#2}
	\aliascntresetthe{#1}
	\expandafter\def\csname #1autorefname\endcsname{#2}
}

\theoremstyle{plain}
\mynewtheorem{thm}{Theorem}
\mynewtheorem{thmdef}{Theorem/Definition}
\mynewtheorem{prop}{Proposition}
\mynewtheorem{cor}{Corollary}
\mynewtheorem{lem}{Lemma}
\mynewtheorem{conj}{Conjecture}
\theoremstyle{definition}
\mynewtheorem{defn}{Definition}
\mynewtheorem{expl}{Example}
\mynewtheorem{prob}{Problem}
\mynewtheorem{ques}{Question}
\mynewtheorem{cond}{Condition}
\mynewtheorem{conv}{Convention}
\mynewtheorem{defthm}{Definition/Theorem}
\theoremstyle{remark}
\mynewtheorem{rem}{Remark}

\usepackage{shuffle}

\newcommand{\oshuffle}{\overline{\shuffle}}
\newcommand{\osh}{\overline{\mathrm{Sh}}}
\newcommand{\sh}{\mathrm{Sh}}

\newcommand{\bfX}{\boldsymbol{X}}
\newcommand{\bfY}{\boldsymbol{Y}}
\newcommand{\bfH}{\overline{H}}
\newcommand{\bfK}{\overline{K}}
\newcommand{\Dshuffle}{\Delta_{\shuffle}}
\newcommand{\wDshuffle}{\widetilde\Delta_{\shuffle}}
\newcommand{\wDOshuffle}{\widetilde\Delta_{\overline \shuffle}}
\newcommand{\DOshuffle}{\Delta_{\overline \shuffle}}
\newcommand{\Dtensor}{\Delta_{\otimes}}

\newcommand{\boldi}{{\boldsymbol i}}
\newcommand{\boldj}{{\boldsymbol j}}
\newcommand{\boldk}{{\boldsymbol k}}
\newcommand{\boldh}{{\boldsymbol h}}
\newcommand{\bolda}{{\boldsymbol \alpha}}
\newcommand{\boldb}{{\boldsymbol \beta}}
\newcommand{\boldc}{{\boldsymbol \gamma}}
\newcommand{\boldd}{{\boldsymbol \delta}}

\newcommand{\p}{{\lfloor p \rfloor}}
\newcommand{\xhapprox}{\upharpoonleft_{\bfX} \!\! \bfH}
\newcommand{\xhtrue}{\uparrow_{\bfX} \!\! \bfH}
\newcommand{\edif}{\mathrm{d}}
\newcommand{\dif}{\mathrm{d}}

\newcommand{\bbR}{\mathbb R}

\usepackage{scalerel}

\usepackage{authblk}

\usepackage[pdftex, ocgcolorlinks, colorlinks = true]{hyperref}

\makeatletter
\renewcommand*{\eqref}[1]{%
	\hyperref[{#1}]{\textup{\tagform@{\ref*{#1}}}}%
}
\makeatother

\allowdisplaybreaks[1]

\begin{document}

\date{\today}

\title{A combinatorial approach to geometric rough paths and their controlled paths}
\date{\today}
\author[1]{Thomas Cass\thanks{The research of the first named author is supported by EPSRC Programme Grant EP/S026347/1.}}
\author[2]{Bruce K.\ Driver}
\author[3]{Christian Litterer\thanks{The research of the third named author was partially supported by EPSRC grant EP/V005413/1.}}
\author[1]{Emilio Rossi Ferrucci\thanks{The fourth named author's PhD is funded by the Centre for Doctoral Training in Financial Computing \& Analytics.}\thanks{Corresponding author: \href{mailto:emilio.rossi-ferrucci16@imperial.ac.uk}{\nolinkurl{emilio.rossi-ferrucci16@imperial.ac.uk}}}}
\affil[1]{\small Dept.\ of Mathematics, Imperial College London}
\affil[2]{\small Dept.\ of Mathematics, University of California San Diego}
\affil[3]{\small Dept.\ of Mathematics, University of York}
\date{\today}
\maketitle

\abstract{We develop the structure theory for transformations of weakly geometric
	rough paths of bounded $1 < p$-variation and their controlled paths. Our approach differs from existing approaches as it does not rely on smooth
	approximations. We derive an explicit combinatorial expression
	for the rough path lift of a controlled path, and use it to obtain fundamental identities such as the associativity of the rough integral, the adjunction between pushforwards and pullbacks, and a change of variables formula for rough differential equations (RDEs). As applications we define rough paths, rough integration and RDEs on manifolds, extending the results of \cite{CDL15} to the case of arbitrary $p$.\\[0.5em] MSC classes: 60L20}\\

\section*{Introduction}
The theory of rough paths and the fundamental ideas that underlie it
are now well established, with a vast number of applications in areas as
diverse as stochastic and numerical analysis, machine learning and
stochastic partial differential equations. Martin Hairer's celebrated work on
regularity structures, for example, can be regarded as a far-reaching
generalisation that is inspired by some of the fundamental ideas
underlying rough path theory (and indeed rough paths can be identified with
a special case of the theory). The basic theory is now somewhat classical
and several approaches have been developed to obtain (and in some cases
significantly extend) its core results \cite{Lyo98} \cite{LQ02} \cite{Gub04} \cite{LCL07} \cite{Gub10} \cite{FV10} \cite{FH14}. In these approaches, the core results are usually either developed explicitly for the case $2\leq p<3$ or follow from properties inherited by
approximation arguments from the corresponding properties of the lifts of smooth paths.

In this paper we use algebraic and combinatorial
methods to explore the basic structure underlying transformations of rough paths of
arbitrary roughness under sufficiently regular maps. Our approach allows us
to work directly with weakly geometric rough paths and leads to a clean separation
of analysis and algebra, yielding explicit combinatorial
descriptions of the resulting objects. The structure theory for geometric rough paths is usually
deduced from the corresponding properties of the (lifts of) smooth paths by
taking closures in a suitable rough path metric. In finite dimensions, this means that identities for geometric rough paths readily extend to the weakly
geometric setting. However, this extension is predicated on the close
relation of weakly geometric and geometric rough paths established by Friz
and Victoir \cite{FV10} for paths with values in finite-dimensional spaces. It
is presently not clear if a similar relation holds in infinite dimensions.
Similarly, the smooth approximation arguments are not available when
studying more general branched rough paths. Although the combinatorics is, at times, quite complex, we have made it a priority to state all results in a clear, coordinate-free manner, without resorting to the coordinate notation that is used in the proofs.

Another goal of this paper is to unify Lyons's original approach and Gubinelli's \say{linearised} version \cite{Gub04} which deals with controlled rough paths, or \emph{controlled paths} as we call them here, to avoid ambiguity. This is widely considered to be the most general and modern approach to rough path theory. Many fundamental results, such as the definition and convergence of controlled-rough integrals, are not present in the literature, stated in this setting. The fact that we work with controlled paths provides further motivation to avoid smooth approximation: controlled paths are only indirectly defined in terms of their reference rough path, and a smooth approximation of the rough path does not automatically yield one of the controlled path.

In \cite{CDL15} the authors derive, without using smooth approximation arguments,
the basic structure theory for weakly geometric $p-$rough paths for the case 
$2\leq p<3$, leading to an explicit characterisation of rough paths
constrained to an embedded submanifold of $\bbR^d$. In
this paper we generalise these results to weakly
geometric rough paths of arbitrary roughness. Even though our results
are proved in the finite-dimensional setting thanks to the use of bases, the statements themselves are stated in a coordinate free manner and preserve their meaning when dropping the finite-dimensional assumption. We therefore expect that with some care they may be generalised to the infinite-dimensional setting, and hope that they will inform future work on branched rough paths. Generalising the arguments from the $2 \leq p < 3$ case is
a non-trivial challenge, as identities and proofs become significantly more
combinatorial in nature. As applications we develop the theory of weakly geometric rough paths on manifolds, their controlled paths, and the associated integration theory and rough differential equations (RDEs).
For arbitrary $p$-rough paths the combinatorial identities corresponding to
key structural properties of rough paths such as the functoriality of
pushforwards are complex and difficult to obtain directly. Fortunately,
analogous results for controlled paths are more readily obtained
exploiting their more linear structure. Consequently, we first establish
identities for controlled paths and then deduce the results for full
rough paths. It is a well-known fact going back to \cite{Gub04} that a controlled path can in principle be lifted
to a full rough path. An explicit construction for the case $2\leq p<3$ can for
example be found in \cite{FH14}. For the general lift of a controlled rough
path in the branched case Gubinelli has obtained an inductive formula \cite[Remark 8.7]{Gub10}. Unfortunately, the case $2\leq p<3$ is not very
representative for the construction of the lift and the description of the
lifted rough path in \cite{Gub10} is not very explicit. In this paper we
give an closed-form description of the lift for general $p$. The result may be
regarded as a generalisation of the construction of the rough integral
carried out in \cite{LCL07}. However, our result does
not require any of the symmetries of the one-form integrands, and instead uses a combinatorial property of the ordered shuffle. In addition,
generalising an argument from \cite{CDLL16}, we are able show that a controlled
rough path defined with respect to a weakly geometric rough path lifts again
to a weakly geometric path.

The paper is structured as follows: in \autoref{sec:alg} we introduce
algebraic preliminaries and notations, in particular the shuffle and ordered shuffle
products, and associated lemmata that will be integral to our constructions. In \autoref{sec:grps} we recall the basic notion of weakly geometric and controlled paths and proceed to construct the lift of a controlled path for general $p$. Underlying the proof is a
generalisation of an argument in \cite{CDLL16} that gave a similar result for
Lip-$\gamma$ 1-forms, but avoids the use of symmetries not
present in general Gubinelli derivatives. The construction of the lift for general $p$ has
considerable combinatorial complexity but enables us to define pushforwards of
rough paths using controlled paths. In order to study the properties of
those pushforwards we define a change of reference rough path for a controlled
rough path and study its structure. The machinery developed for controlled
rough paths allows us in \autoref{prop:properties} to establish the
fundamental properties of pushforwards of $p-$rough paths. In particular, we
prove the functorial property of the pushfowards. In \autoref{thm:assoc} we
demonstrate the associativity of the rough integral, i.e.\ the rigorous formulation of the heuristic statement that \say{we can substitute differentials}. \autoref{thm:pushPull}
establishes a pushforward-pullback adjunction of controlled paths and
their reference rough paths under the rough integral pairing. Finally, combining these results in \autoref{thm:changeVarRDEs} we obtain a change of variable result for the solutions of
rough differential equations (RDEs). All proofs separate algebra and
analysis, in that they do not require approximation by lifts of smooth paths.
In \autoref{sec:mfds} we study applications to rough paths on manifolds. The
structure of pushforwards of rough path gives rise to a natural notion of
rough paths on manifolds defined locally using the chart functions (c.f.\
\cite{BL15}). The change of variable
result for RDEs in \autoref{thm:changeVarRDEs} motivates a definition of
RDEs on manifolds. We clarify some of the relations of the intrinsic rough
paths considered in this paper with existing notions of rough paths that
have been obtained in different geometric settings. The interplay between
controlled paths and the change of their reference path allows us to
describe a natural class of controlled paths with respect to a rough
path on a manifold. We define the rough integral for a suitable class of
integrands controlled by a rough path on a manifold, and give meaning to an RDE driven by a manifold-valued rough path, with solution valued in a second manifold. In \autoref{rem:extrinsic} we show how our definitions directly extend those of \cite{CDL15} in the extrinsic framework.

\newcommand\DrawControl[3]{
	node[#2,circle,fill=#2,inner sep=2pt,label={above:$#1$},label={[black]below:{\footnotesize#3}}] at #1 {}
}

\usetikzlibrary{decorations.pathreplacing}

\section{The shuffle and tensor bialgebras}\label{sec:alg}

We begin with a concise review of bialgebras defined on tensor algebras, for which we refer to \cite[Chapter I]{Man06} \cite[Chapter 2]{Wei18}. Given $n_1,\ldots,n_m \in \mathbb N$ (which may be 0) we define $\sh(n_1,\ldots,n_m)$ to be the subset of the permutation group $\mathfrak S_{n_1 + \ldots + n_m}$ of $(n_1,\ldots,n_m)$-\emph{shuffles}, i.e.\ permutations $\sigma$ with the property that 
\begin{equation}\label{eq:shufflePerm}
	\sigma(n_1 + \ldots + n_{i-1} + 1) < \sigma(n_1 + \ldots + n_{i-1} + 2) < \ldots < \sigma(n_1 + \ldots + n_i)
\end{equation}
for $i = 1,\ldots,m$ (with $n_0 \coloneqq 0$). We will additionally call $\sigma$ an $(n_1,\ldots,n_m)$-\emph{ordered shuffle} if 
\begin{equation}\label{eq:oShufflePerm}
	\sigma(n_1) \leq \sigma(n_1 + n_2) \leq \ldots \leq \sigma(n_1 + \ldots + n_m)
\end{equation}
and we denote the set of these with $\osh(n_1,\ldots,n_m)$. If $n_i = 0$ for some $i$ we have $\sh(n_1,\ldots,n_m) = \sh(n_1,\ldots, \widehat n_i, \ldots,n_m)$, $\osh(n_1,\ldots,n_m) = \osh(n_1,\ldots, \widehat n_i, \ldots,n_m)$ (with $\widehat{\phantom{n}}$ denoting omission).

In this paper the letters $U,V,W,\ldots$ will always be finite-dimensional $\bbR$-vector spaces. Given such a vector space $V$, a permutation $\sigma \in \mathfrak S_n$ induces a linear isomorphism
\begin{equation}
	\sigma_* \colon V^{\otimes n} \to V^{\otimes n}, \quad v_1 \otimes \cdots \otimes v_n \mapsto v_{\sigma(1)} \otimes \cdots \otimes v_{\sigma(n)}.
\end{equation}
If $\{e_i\}_{i \in I}$ is a basis of $V$ we may write $a \in V^{\otimes n}$ uniquely as $a^{(i_1,\ldots,i_n)} e_{i_1} \otimes \cdots \otimes e_{i_n} \eqqcolon a^{\boldi} e_{\boldi}$ where there is a sum on the ordered $n$-tuple of basis indices $\boldi \coloneqq (i_1,\ldots,i_n)$ (following the Einstein convention), and
\[
\sigma_*(a) = a^\boldi e_{\sigma_*\boldi} = a^{\sigma^{-1}_*\boldj} e_{\boldj}
\]
where 
\begin{equation}\label{eq:pushfwdTuple}
	\rho_*\boldi \coloneqq (i_{\rho(1)}, \ldots, i_{\rho(n)}) \text{ for } \rho \in \mathfrak S_n.
\end{equation}
We let $I^\bullet \coloneqq \bigcup_{n \in \mathbb N} I^n$ be the set of $I$-valued tuples; this includes the empty tuple $()$, and we use $I^\bullet_*$ to denote the set of all such non-empty tuples. We will sometimes identify a tuple $(k_1,\ldots,k_n)$ with the corresponding tensor $e_{k_1} \otimes \cdots \otimes e_{k_n}$ according to the chosen basis. Note that $\sigma^{-1}_* \boldi$ is the tuple obtained by \say{permuting $\boldi$ according to $\sigma$}, e.g.\ if 
\[
\sigma = \begin{pmatrix} 1 & 2 & 3 & 4 & 5 \\ 1 & 3 & 5 & 2 & 4
\end{pmatrix} \in \sh(3,2), \quad \sigma^{-1} = \begin{pmatrix} 1 & 2 & 3 & 4 & 5 \\ 1 & 4 & 2 & 5 & 3
\end{pmatrix} \in \mathfrak S_5
\]
then
\[
\sigma_*^{-1}(i_1,i_2,i_3,i_4,i_5) = (i_1,i_4,i_2,i_5,i_3).
\]
Note that the composition rule (both for tensors and tuples) is 
\begin{equation}\label{eq:compRule}
	(\sigma \circ \rho)_* = \rho_* \sigma_*.
\end{equation}
Indeed, denoting $w_k \coloneqq v_{\sigma(k)}$ we have
\begin{align*}
	(\sigma \circ \rho)_*(v_1 \otimes \cdots \otimes v_n) &= v_{\sigma(\rho(1))} \otimes \cdots \otimes v_{\sigma(\rho(n))} \\
	&= w_{\rho(1)} \otimes \cdots \otimes w_{\rho(n)} \\
	&= \rho_*(w_1 \otimes \cdots \otimes w_n) \\
	&= \rho_* \sigma_* (v_1 \otimes \cdots \otimes v_n).
\end{align*}
For a tuple $\boldi = (i_1, \ldots, i_n)$ we will denote $|\boldi| \coloneqq n$ its length, and given two tuples $\boldi, \boldj$ we write $\boldi \boldj$ for their concatenation. We will denote $T(V) \coloneqq \bigoplus_{n = 0}^\infty V^{\otimes n}$ and for $a \in T(V)$, with $a^n$ its projection onto $V^{\otimes n}$; this has a distinct meaning to the notation $a^{\boldi}$ for tuples $\boldi \in I^\bullet$, explained above. When we are considering the tensor products of the dual $V^*$ of a vector space $V$, or more generally the space of linear maps $\mathcal L(V,W)$ from $V$ to another vector space $W$, we will replace superscripts with subscripts and vice-versa.

We denote by $(T(V),\otimes, \Delta_{\shuffle})$ the tensor bialgebra of $V$, i.e.\ the product is given by the ordinary tensor product, which in coordinates reads
\begin{equation}
	(a \otimes b)^{\boldk} = \sum_{\boldi \boldj = \boldk} a^{\boldi}b^{\boldj}
\end{equation}
and the \emph{shuffle coproduct} is defined on elementary tensors (and extended linearly) as
\begin{equation}\label{eq:Dshuf}
	\begin{split}
		\Dshuffle \colon T(V) &\to T(V) \boxtimes T(V) \\ v_1 \otimes \cdots \otimes v_n &\mapsto \sum_{\substack{k = 0,\ldots, n \\ \sigma \in \sh(k,n-k)}} (v_{\sigma(1)} \otimes \cdots \otimes v_{\sigma(k)}) \boxtimes (v_{\sigma(k+1)} \otimes \cdots \otimes v_{\sigma(n)}).
	\end{split}
\end{equation}
Here use the symbol $\boxtimes$ to denote external tensor product, reserving $\otimes$ for the algebra product. In coordinates this reads, for $a \in T(V)$
\begin{equation}\label{eq:DshuffleCoords}
	(\Dshuffle a)^{\boldi, \boldj} = \sum_{\substack{\sigma \in \sh(|\boldi|,|\boldj|) \\ \boldk = \sigma^{-1}_*(\boldi \boldj)} } a^{\boldk} \eqqcolon \sum_{\boldk \in \sh(\boldi, \boldj)} a^{\boldk}
\end{equation}
where $\boldi$ denotes the index of the first $\boxtimes$-factor and $\boldj$ that of the second. In order to give a precise meaning to $\sh(\boldi, \boldj)$ we must introduce multiset notation. Recall that a multiset is like a set (in that the order of its elements is not taken into account), but with the difference that the same element may appear more than once; we will denote multisets with double braces, e.g.\ $\{\!\{1,2,2,2,3,3\}\!\}$. If $A$ and $B$ are multisets we write $A \subseteq B$ if each element of $A$ belongs to $B$ counted with its multiplicity, e.g.\ $\{\!\{2,2,3,3\}\!\} \subseteq \{\!\{1,2,2,2,3,3\}\!\}$ but $\{\!\{1,1,2,2,3,3\}\!\} \not\subseteq \{\!\{1,2,2,2,3,3\}\!\}$, and $A = B$ is defined to mean $A \subseteq B$ and $B \subseteq A$. With this in mind, we are defining
\begin{equation}
	\sh(\boldi, \boldj) \coloneqq \{\!\{ \sigma^{-1}_*(\boldi \boldj) \mid \sigma \in \sh(|\boldi|,|\boldj|)\}\!\}.
\end{equation}
This means that the tuple $\boldk$ appears as many times as there are $\sigma$'s with the property that $\boldk = \sigma^{-1}_*(\boldi \boldj)$. Similar multisets will be defined without explicit mention from now on. 

The bialgebra that is graded dual to $(T(V),\otimes, \Delta_{\shuffle})$ is given by $(T(V^*),\shuffle,\Delta_\otimes)$. Note that we are using the notion of graded duality for bialgebras (see \cite[\S 1.5]{Foi13}), which is different to ordinary duality: in a nutshell, this just means that we are taking the dual of each (finite-dimensional) direct summand and that the product (coproduct) in one bialgebra is the dual to the product (coproduct) in the other; \say{dual} here makes sense because products and coproducts respect the grading. This allows us to avoid considering formal series of tensors, which are unnecessary when considering rough paths without their full signatures, and retain most of the usual properties of duality. $\shuffle = \Dshuffle^*$ is the \emph{shuffle product} given by
\begin{equation}\label{eq:shuf}
	\begin{split}
		\shuffle \colon T(V^*) \boxtimes T(V^*) &\to T(V^*), \\ (v^1 \otimes \cdots \otimes v^n) \boxtimes (v^{n+1} \otimes \cdots \otimes v^{n+m}) &\mapsto \sum_{\sigma \in \sh(n,m)}v^{\sigma^{-1}(1)} \otimes \cdots \otimes v^{\sigma^{-1}(n+m)}
	\end{split}
\end{equation}
which in coordinates (using subscripts, since we are working in $V^*$) reads
\begin{equation}
	(a \shuffle b)_\boldk = \sum_{\substack{|\boldi| = 0,\ldots, |\boldk| \\ \sigma \in \sh(|\boldi|, |\boldk| - |\boldi|) \\ \boldi \boldj = \sigma_*\boldk}} a_\boldi b_\boldj \eqqcolon \sum_{(\boldi, \boldj) \in \sh^{-1}(\boldk)} a_\boldi b_\boldj.
\end{equation}
Note that with this notation $(\boldi, \boldj) \in \sh^{-1}(\boldk) \Leftrightarrow \boldk \in \sh(\boldi\boldj)$; in particular $\boldi$ or $\boldj$ may be empty. The coproduct $\Dtensor = \otimes^*$ is the \emph{deconcatenation coproduct} given by
\begin{equation}
	\begin{split}
		\Dtensor \colon T(V^*) &\to T(V^*) \boxtimes T(V^*) \\ v^1 \otimes \cdots \otimes v^n &\mapsto \sum_{k=0}^n (v^1 \otimes \cdots \otimes v^k) \boxtimes (v^{k+1} \otimes \cdots \otimes v^n)
	\end{split}
\end{equation}
or in coordinates
\begin{equation}
	(\Dtensor a)_{\boldi, \boldj} = a_{\boldi \boldj}.
\end{equation}
Recall that in every bialgebra with coproduct $\Delta$ we may define its reduced coproduct $\widetilde \Delta a \coloneqq \Delta a - a \boxtimes 1 - 1 \boxtimes a$, which is also coassociative. Also recall that in a coalgebra $(C,\Delta)$ (for us $C$ will always be a tensor algebra) the (reduced) coproduct can be iterated, as coassociativity guarantees that $\Delta^m \colon C \to C^{\boxtimes m}$ has a unique meaning (note that under this convention $\Delta^2 = \Delta$, $\Delta^1 \coloneqq \mathbbm 1_C$, $\Delta^0 \coloneqq 1_\mathbb{R}$). Since the above bialgebras are graded and connected the reduced (iterated) coproduct factors as 
\begin{equation}
	\widetilde \Delta^m = \pi_{\geq 1}^{\boxtimes m} \circ \Delta^m
\end{equation}
where $\pi_{\geq 1} \colon T(V) \twoheadrightarrow \bigoplus_{n \geq 1} V^{\otimes n}$ is the projection onto tensor products of positive order.

We may define the \emph{ordered shuffle coproduct} $\Delta_{\oshuffle}$ and the \emph{ordered shuffle product} $\oshuffle$ by requiring shuffles in \eqref{eq:Dshuf} and \eqref{eq:shuf} to be ordered. This does not, in fact, define a real (co)product, because $\oshuffle$ fails to be associative: indeed, it satisfies the alternative relation
\[
a \oshuffle (b \oshuffle c) = (a \oshuffle b + b \oshuffle a) \oshuffle c.
\]
This property makes $(T(V), \oshuffle)$ a Zinbiel algebra \cite{E-FF15}. Whenever we iterate $\oshuffle$ or $\DOshuffle$ we will be carrying out composition from left to right, i.e.\ inductively
\begin{equation}
	\begin{split}
		a_1 \oshuffle \cdots \oshuffle a_n &\coloneqq (a_1 \oshuffle a_2) \oshuffle a_3 \oshuffle \cdots \oshuffle a_n \\
		\DOshuffle^m a &\coloneqq \sum_{(a)^m_{\oshuffle}} a_{(1)} \boxtimes \cdots \boxtimes a_{(m)} \\
		&\coloneqq \sum_{(a)_{\oshuffle}^{m-1}} \Big( \DOshuffle a_{(1)} \Big) \boxtimes a_{(2)} \boxtimes \cdots \boxtimes a_{(m-1)}.
	\end{split}
\end{equation}
This guarantees that the coordinate expression provided for the unordered shuffle carries over to the ordered case, with $\osh$ instead of $\sh$, e.g.\
\begin{equation}\label{eq:wDOshuffleCoords}
	\begin{split}
		(\wDOshuffle^m a)^{\boldk^1,\ldots,\boldk^m} &= [|\boldk^1|,\ldots,|\boldk^m| \geq 1] \sum_{\substack{\sigma \in \osh(|\boldk^1|,\ldots,|\boldk^m|) \\ \boldk = \sigma^{-1}_*(\boldk^1\ldots\boldk^m)} } a^{\boldk} \\
		&\eqqcolon [|\boldk^1|,\ldots,|\boldk^m| \geq 1]\sum_{\boldk \in \osh(\boldk^1,\ldots, \boldk^m)} a^{\boldk}.
	\end{split}
\end{equation}
Here the square bracket has binary value depending on the truth value of the proposition it contains, and is present because the coproduct is reduced; the set of permutations over which the sum is taken is given by ordered shuffles, reflecting the fact that we are dealing with the ordered shuffle coproduct.

Before proceeding, we take a moment to motivate our interest in shuffles and ordered shuffles, although this will become much clearer in \autoref{sec:grps}. While it is well-known that the former are used to express products of iterated integrals of a (smooth) path $X$
\begin{align*}
	&\mathrel{\phantom{=}}\bigg( \int_{s < u_1 < \ldots < u_n < t} \dif X^{i_1}_{u_1} \cdots \dif X^{i_m}_{u_m} \bigg) \bigg( \int_{s < v_1 < \ldots < v_n < t} \dif X^{j_1}_{v_1} \cdots \dif X^{j_n}_{v_n} \bigg) \\
	&= \sum_{\boldk \in \sh(\boldi, \boldj)} \int_{s < r_1 < \ldots < r_{n+m} < t} \dif X^{k_1}_{r_1} \cdots \dif X^{k_{m+n}}_{r_{m+n}},
\end{align*}
the role of ordered shuffles in the study of iterated integrals of paths is less appreciated. One way to motivate their significance is as follows: let $Y$ be the solution to the ODE
\[
\dif Y = V(Y) \dif X = V(Y)\dot X \dif t.
\]
with $X$ $V$-valued and $Y$ $W$-valued. We fix bases on both vector spaces and use Greek indices for $V$ and Latin ones for $W$ - this will be the convention later on as well. Substituting formal Euler expansions, and defining $V_{\boldc}^k(y) \coloneqq V_{\gamma_1}\cdots V_{\gamma_n}^k(y)$ for a tuple $\boldc = (\gamma_1, \cdots, \gamma_n)$, with the product denoting iterated composition of vector fields (i.e.\ $V_\gamma f(y) \coloneqq \partial_k f(y) V^k_\gamma(y)$ for a function $f$ of $y$)
\begin{align}\label{eq:liftMotivation}
	&\mathrel{\phantom{=}}\int_{s < u_1 < \ldots < u_m < t} \dif Y_{u_1}^{k_1} \cdots \dif Y_{u_m}^{k_m} \notag \\
	&= \int_{s < u_1 < \ldots < u_m < t} \dif \Big( V_{\boldc^1}^{k_1}(Y_s) \bfX^{\boldc^1}_{su_1} \Big) \cdots \dif \Big( V_{\boldc^m}^{k_m}(Y_s) \bfX^{\boldc^m}_{su_m} \Big) \notag \\
	&=V_{\boldc^1}^{k_1}(Y_s) \cdots  V_{\boldc^m}^{k_m}(Y_s)\int_{s < u_1 < \ldots < u_m < t} \dif({\textstyle \int_{s < r^1_1 < \ldots < r^1_{n_{\scaleto{1}{2.5pt}}} < u_n}}\dif X_{r^1_1}^{\gamma^1_1}\cdots \dif X_{r^1_{n_{\scaleto{1}{2.5pt}}}}^{\gamma^1_{n_{\scaleto{1}{2.5pt}}}})\cdots \notag \\
	&\mathrel{\phantom{=}}  \cdots \dif({\textstyle \int_{s < r^m_1 < \ldots < r^m_{n_{\scaleto{m}{1.8pt}}} < u_m}}\dif X_{r^m_1}^{\gamma^m_1}\cdots \dif X_{r^m_{n_{\scaleto{m}{1.8pt}}}}^{\gamma^m_{n_{\scaleto{m}{1.8pt}}}}) \notag \\
	&=  V_{\boldc^1}^{k_1}(Y_s) \cdots  V_{\boldc^m}^{k_m}(Y_s) \\
	&\mathrel{\phantom{=}} \cdot \int_{\substack{s < r^1_1 < \ldots < r^1_{n_{\scaleto{1}{2.5pt}} - 1} < r^1_{n_{\scaleto{1}{2.5pt}}}  \\ \vdots \\ s < r^m_1 < \ldots < r^m_{n_{\scaleto{m}{1.8pt}} - 1} < r^m_{n_{\scaleto{m}{1.8pt}}} \\ s < r^1_{n_{\scaleto{1}{2.5pt}}} < \ldots < r^m_{n_{\scaleto{m}{1.8pt}}} < t}} \dif X_{r^1_1}^{\gamma^1_1}\cdots \dif X_{r^1_{n_{\scaleto{1}{2.5pt}}}}^{\gamma^1_{n_{\scaleto{1}{2.5pt}}}} \cdots \ \cdots \dif X_{r^m_1}^{\gamma^m_{n_{\scaleto{1}{2.5pt}}}}\cdots \dif X_{r^m_{n_{\scaleto{m}{1.8pt}}}}^{\gamma^m_{n_{\scaleto{m}{1.8pt}}}}\notag \\
	&= \sum_{\boldc^1,\ldots,\boldc^m} V_{\boldc^1}^k(Y_s) \cdots  V_{\boldc^m}^k(Y_s)\sum_{\boldc \in \osh(\boldc^1,\ldots,\boldc^m)} \int_{s < v_1 < \ldots < v_{n \coloneqq \sum_i \! n_{\scaleto{i}{2.5pt}}}} \dif X^{\gamma_1}_{v_1} \cdots \dif X^{\gamma_n}_{v_n}. \notag
\end{align}
In other words, ordered shuffles index the sum involved in the expression of the iterated integrals of $Y$ in terms of those of $X$. The relevance of ordered shuffles in the similar cases of linear RDEs and Lip$(\gamma-1)$ functions was observed in \cite[p.72-75]{LCL07}.\\

We will now prove a few results that will be used in \autoref{sec:grps}. These will be stated in terms of tuples (although they are essentially statements about shuffles), so they can be readily deployed when dealing with rough paths. Unshuffling and concatenating satisfy the following commutativity relation:
\[
\begin{tikzcd}[column sep = huge]
	T(V) \arrow[r,"\Dshuffle^m"] \arrow[d, leftarrow, "\otimes"] &T(V)^{\boxtimes m} \arrow[d, leftarrow,"\otimes_{1,m+1} \boxtimes \cdots \boxtimes \otimes_{m,2m}"] \\
	T(V)^{\boxtimes 2} \arrow[r,"\Dshuffle^m \boxtimes \Dshuffle^m"] &T(V)^{\boxtimes 2m}
\end{tikzcd}.
\]
The following lemma can be viewed as the counterpart to this statement in the context of ordered shuffles. Since it is essentially combinatorial in nature, we will state it in terms of tuples. Recall that $I^\bullet_*$ denotes the set of all $I$-valued tuples of any positive order.
\begin{lem}\label{lem:orderedCompat} Let $\boldk^1,\ldots,\boldk^m \in I^\bullet_*$. The following identity of multisets holds:
	\begin{equation}\label{eq:orderedCompat}
		\begin{split}
			&\mathrel{\phantom{=}}\{\!\{(\boldi, \boldj) \in I^\bullet \times I^\bullet \mid \boldi \boldj \in \overline{\emph{Sh}}(\boldk^1,\ldots,\boldk^m)  \}\!\} \\
			&= \{\!\{(\boldi, \boldj) \in I^\bullet \times I^\bullet \mid \boldj \in \overline{\emph{Sh}}(\boldj^{l+1},\ldots,\boldj^m);\\
			&\mathrel{\phantom{=}}\boldi \in \emph{Sh}(\overline{\emph{Sh}}(\boldi^1,\ldots,\boldi^l),\emph{Sh}(\boldi^{l+1},\ldots,\boldi^m)); \\
			&\mathrel{\phantom{=}}\emph{where }\boldi^h = \boldk^h \emph{ for } h \leq l;\\
			&\mathrel{\phantom{=}} \boldi^h \boldj^h = \boldk^h \emph{ and } |\boldj^h| \geq 1, \emph{ for } h \geq l+1; \\
			&\mathrel{\phantom{=}}\emph{with } l = 0,\ldots,m \}\!\}.
		\end{split}
	\end{equation}
\end{lem}

The following picture is meant to illustrate the idea of the statement: the horizontal lines represent tuples, and the red bullet points represent their terminal elements.

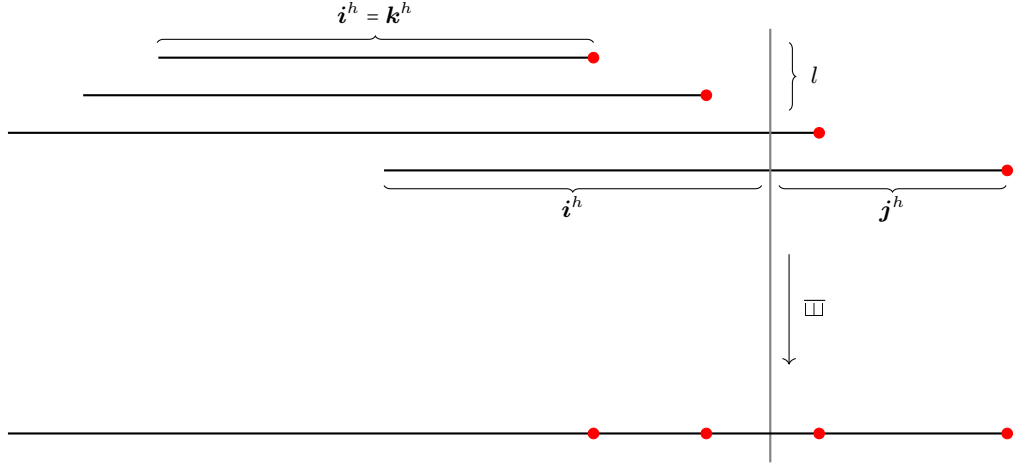
\begin{center}
	\begin{tikzpicture}
		\tikzset{
			position label/.style={
				below = 3pt,
				text height = 2ex,
				text depth = 1ex
			}
		}
		\tikzset{
			position label/.style={
				below = 3pt,
				text height = 1.5ex,
				text depth = 1ex
			},
			brace/.style={
				decoration={brace, mirror},
				decorate
			}
		}
		
		\tikzset{
			position label/.style={
				below = 3pt,
				text height = 1.5ex,
				text depth = 1ex
			},
			braceup/.style={
				decoration={brace},
				decorate
			}
		}
		
		\node (x01) at (-1,0){};
		\node (x02) at (5,0){};
		\node (x11) at (-1,-0.5){};
		\node (x12) at (5,-0.5){};
		\node (x21) at (-2,-1){};
		\node (x22) at (6.5,-1){};
		\node (x31) at (-3,-1.5){};
		\node (x32) at (8,-1.5){};
		\node (x41) at (2,-2){};
		\node (x42) at (10.5,-2){};
		\node (vline1) at (7.25,0){};
		\node (vline2) at (7.25,-6){};
		\node (ibracket1) at (2,-2.2){};
		\node (ibracket2) at (7.25,-2.2){};
		\node (jbracket1) at (7.25,-2.2){};
		\node (jbracket2) at (10.5,-2.2){};
		
		\node (x51) at (-3,-5.5){};
		\node (x52) at (5,-5.5){};
		\node (x53) at (6.5,-5.5){};
		\node (x54) at (8,-5.5){};
		\node (x55) at (10.5,-5.5){};
		
		\node (arrow1) at (7.5,-3){};
		\node (arrow2) at (7.5,-4.7){};

		\draw [braceup] (7.5,-0.3) -- (7.5,-1.2) node [position label, pos=0.5, yshift=2.5ex, xshift=2.5ex] {$l$};

		\draw[black, thick] (x11) -- (x12);
		\filldraw[red] (4.9,-0.5) circle (2pt) node[anchor=north]{};
		\draw [braceup] (-0.9,-0.3) -- (4.9,-0.3) node [position label, pos=0.5, yshift=5ex] {$\boldi^h=\boldk^h$};
		
		\draw[black, thick] (x21) -- (x22);
		\filldraw[red] (6.4,-1) circle (2pt) node[anchor=north]{};
		
		\draw[black, thick] (x31) -- (x32);
		\filldraw[red] (7.9,-1.5) circle (2pt) node[anchor=north]{};
		\draw[black, thick] (x41) -- (x42);
		\filldraw[red] (10.4,-2) circle (2pt) node[anchor=north]{};
		
		\draw[gray, thick] (vline1) -- (vline2);
		
		\draw [brace] (ibracket1) -- (ibracket2) node [position label, pos=0.5]{$\boldi^h$};
		\draw [brace] (jbracket1) -- (jbracket2) node [position label, pos=0.5] {$\boldj^h$};

		\draw[black, thick] (x51) -- (x55);
		
		\filldraw[red] (4.9,-5.5) circle (2pt) node[anchor=north]{};
		\filldraw[red] (6.4,-5.5) circle (2pt) node[anchor=north]{};
		\filldraw[red] (7.9,-5.5) circle (2pt) node[anchor=north]{};
		\filldraw[red] (10.4,-5.5) circle (2pt) node[anchor=north]{};
		
		\draw [->] (arrow1) -- (arrow2) node [position label, pos=0.5 , yshift=2.5ex, xshift=2.5ex] {$\overline\shuffle$};

	\end{tikzpicture}
\end{center}

Note that we are taking into account multiplicities in the multiset $\boldk \in \overline{\text{Sh}}(\boldk^1,\ldots,\boldk^m)$, i.e.\ if the same tuple $\boldk$ belongs twice to $\overline{\text{Sh}}(\boldk^1,\ldots,\boldk^m)$, then any pair $(\boldi, \boldj)$ such that $\boldi \boldj = \boldk$ appears twice in the multiset; an analogous remark holds for the right hand side and for all similarly defined multisets. $\mathrm{Sh}(\overline{\mathrm{Sh}}(\boldi^1,\ldots,\boldi^l),\mathrm{Sh}(\boldi^{l+1},\ldots,\boldi^m))$ here stands for \[
\{\!\{ \boldi \in \mathrm{Sh}(\boldsymbol a, \boldsymbol b) \mid \boldsymbol a \in \overline{\mathrm{Sh}}(\boldi^1,\ldots,\boldi^l),\ \boldsymbol b \in \mathrm{Sh}(\boldi^{l+1},\ldots,\boldi^m) \}\!\},
\]
the multiset of tuples obtained by shuffling $\boldi^1,\ldots,\boldi^m$, the first $l$ with order. When $l = 0$ or $m$ this reduces respectively to $\mathrm{Sh}(\mathrm{Sh}(\boldi^{1},\ldots,\boldi^m)) = \mathrm{Sh}(\boldi^{1},\ldots,\boldi^m)$, $\mathrm{Sh}(\overline{\mathrm{Sh}}(\boldi^1,\ldots,\boldi^m)) = \overline{\mathrm{Sh}}(\boldi^1,\ldots,\boldi^m)$, since the only possible way of shuffling a single is to leave it unchanged. The proof of this lemma is most easily understood when going through its steps with reference to the example that immediately follows it.

\begin{proof}[Proof of \autoref{lem:orderedCompat}]
	Let $A(\boldk^1,\ldots,\boldk^m)$ denote the first multiset defined above and $B(\boldk^1,\ldots,\boldk^m)$ the second. For tuples $\boldsymbol \ell^1,\ldots,\boldsymbol\ell^n$ and $1 \leq a \leq b \leq n$ define $\boldsymbol \ell^{a:b} \coloneqq \boldsymbol \ell^a\ldots\boldsymbol\ell^b \in I^\bullet$ (juxtaposition) and for a tuple $\boldsymbol \ell$ and $1 < c \leq d \leq |\boldsymbol\ell|$ $\boldsymbol \ell_{c:d} \coloneqq (\ell_c,\ldots,\ell_d)$, and let $(\boldi, \boldj) \in A(\boldk^1,\ldots,\boldk^m)$. This means there exists $\sigma \in \osh(|\boldk^1|,\ldots,|\boldk^m|)$ with $\boldi\boldj = \sigma^{-1}_*\boldk^{1:m}$. Let 
	\begin{equation*}
		l \coloneqq \begin{cases}
			0 &\text{if } |\boldi| < \sigma(|\boldk^1|) \\
			m &\text{if } \boldj = () \\
			\text{s.t.\ } \sigma(|\boldk^{1:l}|) \leq |\boldi| < \sigma(|\boldk^{1:l+1}|) &\text{otherwise}
		\end{cases}
	\end{equation*}
	which is exists and is unique since $\sigma$ is an ordered shuffle. We then let $\boldi^h \coloneqq \boldk^h$ for $h \leq l$ and for $h \geq l+1$
	\begin{align}\label{eq:ijkhSigma}
		\boldi^h \boldj^h \coloneqq \boldk^h, \quad  \sigma(|\boldk^{1:h-1}| + |\boldi^h|) \leq |\boldi| < \sigma(|\boldk^{1:h-1}| + |\boldi^h| + 1)
	\end{align}
	where $|\boldi^h|$ (and hence $\boldi^h$) is unique since $\sigma$ is a shuffle. Now, it cannot be the case that for $h \geq l + 1$ we have $|\boldi^h| = |\boldk^h|$, for this would violate the definition of $l$: this implies $|\boldj^h| \geq 1$. Moreover, we have $\boldi = \sigma^{-1}_*\boldi^{1:m}$ and $\boldj = \sigma^{-1}_* \boldj^{l+1:m}$, where we are defining the right hand sides using the same expression as before \eqref{eq:pushfwdTuple}, but by considering the numberings on $\boldi^{1:m}$, $\boldj^{l+1:m}$ to be those inherited as subtuples of $\boldk^{1:m}$: this is because $\boldi^h$ occupies the segment of $\boldk^{1:m}$ numbered with $[|\boldk^{1:h-1}| + 1, |\boldk^{1:h-1}| + |\boldi^h|]$, all of which $\sigma$ maps into $[1,|\boldi|]$, by \eqref{eq:ijkhSigma} and again by the shuffle property of $\sigma$; similarly, $\boldj^h$ occupies the segment numbered $[|\boldk^{1:h-1}| + |\boldi^h|+1, |\boldk^{1:h}|]$ which gets mapped above $|\boldi|$. By construction $\sigma$ (once domains are renumbered) shuffles $\boldi^{1:l}$ and $\boldj^{l+1:m}$ with order, since these are the tuples that contain the $k^h_{|\boldk^h|}$'s, and $\boldi^{l+1,m}$ without order. If $\rho \in \mathrm{Sh}(n_1,\ldots,n_m)$ and $S \subseteq \{1,\ldots,n_1 + \ldots + n_m\}$, $\rho|_S$ is still a shuffle, with the additional order constraints on those $(n_1 + \ldots + n_q)$'s that belong to $S$: therefore, we have that $\boldj \in \overline{\mathrm{Sh}}(\boldj^{l+1},\ldots,\boldj^m)$ and $\boldi \in \mathrm{Sh}(\overline{\mathrm{Sh}}(\boldi^1,\ldots,\boldi^l),\mathrm{Sh}(\boldi^{l+1},\ldots,\boldi^m))$. This shows $A(\boldk^1,\ldots,\boldk^m) \subseteq B(\boldk^1,\ldots,\boldk^m)$.
	
	Conversely, let $(\boldi, \boldj) \in B(\boldk^1,\ldots,\boldk^m)$, with $l, \boldi^h, \boldj^h$ as in \eqref{eq:orderedCompat}. $\boldi\boldj$ is obtained by an ordered shuffle of $\boldk^1,\ldots,\boldk^m$: that the order of each $\boldk^h$, $h \leq l$ is preserved is immediate since $\boldi^h = \boldk^h$; that the order of each $\boldk^h$, $h > l$ is preserved is a consequence of the fact that the order of $\boldi^h$ is preserved, that the order of $\boldj^h$ is preserved, and that $\boldi$ comes before $\boldj$ in the juxtaposition $\boldi \boldj$; that the shuffle of the $\boldk^h$'s is ordered is a consequence of the fact that the shuffles of $\boldi^1,\ldots,\boldi^l$ and $\boldj^{l+1},\ldots,\boldj^{l+1}$ are ordered. This shows $B(\boldk^1,\ldots,\boldk^m) \subseteq A(\boldk^1,\ldots,\boldk^m)$; also note that in both inclusions multiplicities are indeed counted, since the correspondence between the underlying permutation in $A(\boldk^1,\ldots,\boldk^m)$ and the pair of underlying permutations in $B(\boldk^1,\ldots,\boldk^m)$ is bijective.
\end{proof}

\begin{expl}
	We illustrate the idea behind this lemma with an example. Let
	\begin{align*}&m = 4; \quad |\boldk^1| = 2,\ |\boldk^2| = 3, \ |\boldk^3| = 4, \ |\boldk^4| = 4 \\
		&\sigma = \begin{pmatrix} 1 & 2 & 3 & 4 & 5 & 6 & 7 & 8 & 9 & 10 & 11 & 12 & 13 \\ 4 & 6 & 3 & 5 & 8 & 7 & 9 & 11 & 12 & 1 & 2 & 10 & 13 \end{pmatrix} \in \osh(2,3,4,4) \\
		&\sigma^{-1}_*(\boldk^1 \boldk^2 \boldk^3 \boldk^4) = (\underbrace{k^4_1,k^4_2,k^2_1,k^1_1,k^2_2,{\color{red}k^1_2},k^3_1,{\color{red}k^2_3}}_{\eqqcolon \boldi}, \underbrace{k^3_2,k^4_3,k^3_3,{\color{red}k^3_4},{\color{red}k^4_4}}_{\eqqcolon \boldj}) \\
		& (\boldi, \boldj) \in A(\boldk^1,\boldk^2,\boldk^3,\boldk^4)
	\end{align*}
	where we have coloured in red the terminal elements of the $\boldk^h$'s (and recall that it is necessary to renumber $\boldk^1\boldk^2\boldk^3\boldk^4$ from $1$ to $13$ before applying \eqref{eq:pushfwdTuple}, and then change the numbering back) and we have
	\begin{align*}
		&(\boldi, \boldj) \in B(\boldk^1,\boldk^2,\boldk^3,\boldk^4) \quad \text{with } l = 2;\\
		&\boldi^1 = (k^1_1,k^1_2), \ \boldi^2 = (k^2_1,k^2_2,k^2_3), \ \boldi^3 = (k^3_1), \ \boldi^4 = (k^4_1,k^4_2); \\
		&\boldj^1 = (k^3_2,k^3_3,k^3_4), \ \boldj^2 = (k^4_3,k^4_4)
	\end{align*}
	since $\boldk^h = \boldi^h$ for $h = 1,2$, $\boldk^h = \boldi^h \boldj^h$ for $h = 3,4$, and
	\begin{align*}
		&\boldj \in \osh(\boldj^1,\boldj^2); \quad \boldi \in \sh(\boldsymbol a, \boldsymbol b)\\
		&\text{with } \boldsymbol a \coloneqq (k^2_1,k^1_1,k^2_2,k^1_2,k^2_3) \in \osh(\boldi^1,\boldi^2) \\
		&\phantom{\text{with }} \boldsymbol b \coloneqq (k^4_1,k^4_2,k^3_1) \in \sh(\boldi^3,\boldi^4)
	\end{align*}
	Note that neither of the two $\sh$'s above can be replaced with $\osh$.
\end{expl}
The reduced version of this \autoref{lem:orderedCompat} would involve restricting it to non-empty tuples $\boldk^1,\ldots,\boldk^m$; we will need the dual of this statement. We will use the notation $((\boldi^1,\ldots, \boldi^l),(\boldi^{l+1},\ldots,\boldi^m)) \in (\osh^{-1},\sh^{-1})(\sh^{-1}(\boldi)) $ as a shorthand for $(\boldi^1,\ldots, \boldi^l) \in \osh^{-1}(\boldsymbol a)$, $(\boldi^{l+1},\ldots,\boldi^m) \in \sh^{-1}(\boldsymbol b) : (\boldsymbol a, \boldsymbol b) \in \sh^{-1}(\boldi)$, and $m$ is fixed.
\begin{cor}\label{cor:dualReduced}
	For $I$-valued tuples $\boldi, \boldj$ the following identity of multisets holds:
	\begin{equation}
		\begin{split}
			&\{\!\{(\boldk^1,\ldots,\boldk^m) \in \overline{\emph{Sh}}^{-1}(\boldi\boldj) \mid |\boldk^1|,\ldots,|\boldk^m| \geq 1\}\!\}\\
			=&\{\!\{(\boldk^1,\ldots,\boldk^m) \in (I^\bullet_*)^m \mid (\boldj^{l+1},\ldots,\boldj^m) \in \overline{\emph{Sh}}^{-1}(\boldj); \\
			&((\boldi^1,\ldots, \boldi^l),(\boldi^{l+1},\ldots,\boldi^m)) \in (\overline{\emph{Sh}}^{-1},\emph{Sh}^{-1})(\emph{Sh}^{-1}(\boldi));\\
			&\boldk^h = \boldi^h, \ h \leq l; \ \boldk^h = \boldi^h \boldj^h, \ |\boldj^h| \geq 1, \ h \geq l+1;\\
			&l = 0,\ldots,m \}\!\}.
		\end{split}
	\end{equation}
	\begin{proof}[Proof of \autoref{cor:dualReduced}]
		Let $C(\boldi, \boldj)$ denote the first multiset above and $D(\boldi, \boldj)$ the second, and recall the names $A(\boldk^1,\ldots,\boldk^m)$, $B(\boldk^1,\ldots,\boldk^m)$ for the sets of \autoref{lem:orderedCompat}. We then have (taking into account multiplicities)
		\begin{align*}
			(\boldk^1,\ldots,\boldk^m) \in C(\boldi, \boldj) &\iff (\boldi, \boldj) \in A(\boldk^1,\ldots,\boldk^m) \\
			&\iff (\boldi, \boldj) \in B(\boldk^1,\ldots,\boldk^m)  \\
			&\iff (\boldk^1,\ldots,\boldk^m) \in D(\boldi, \boldj)
		\end{align*}
		thus concluding the proof.
	\end{proof}
\end{cor}

Next we discuss another combinatorial relation involving ordered and unordered shuffles; similarly to the earlier lemma, we provide a diagram and an example (unrelated to each other) to help explain the idea. $\bigsqcup$ will denote disjoint union of multisets, e.g.\ if the same tuple appears in sets corresponding to two different $\sigma$'s it should be counted twice.
\begin{lem}\label{lem:shsh}
	Let $n \coloneqq n_1+\ldots+n_m$, $n^l \coloneqq n_1 + \ldots + n_l$ for $l = 1,\ldots, m$, and $\boldk^1,\ldots,\boldk^n \in I^\bullet_*$. We have
	\begin{equation}
		\begin{split}
			&\emph{Sh}(\overline{\emph{Sh}}(\boldk^1,\ldots,\boldk^{n_1}), \ldots, \overline{\emph{Sh}}(\boldk^{n^{m-1}+1},\ldots,\boldk^{n})) \\
			=& \bigsqcup_{\pi \in \emph{Sh}(n_1,\ldots,n_m)} \overline{\emph{Sh}}(\boldk^{\pi^{-1}(1)},\ldots,\boldk^{\pi^{-1}(n)}).
		\end{split}
	\end{equation}
\end{lem}

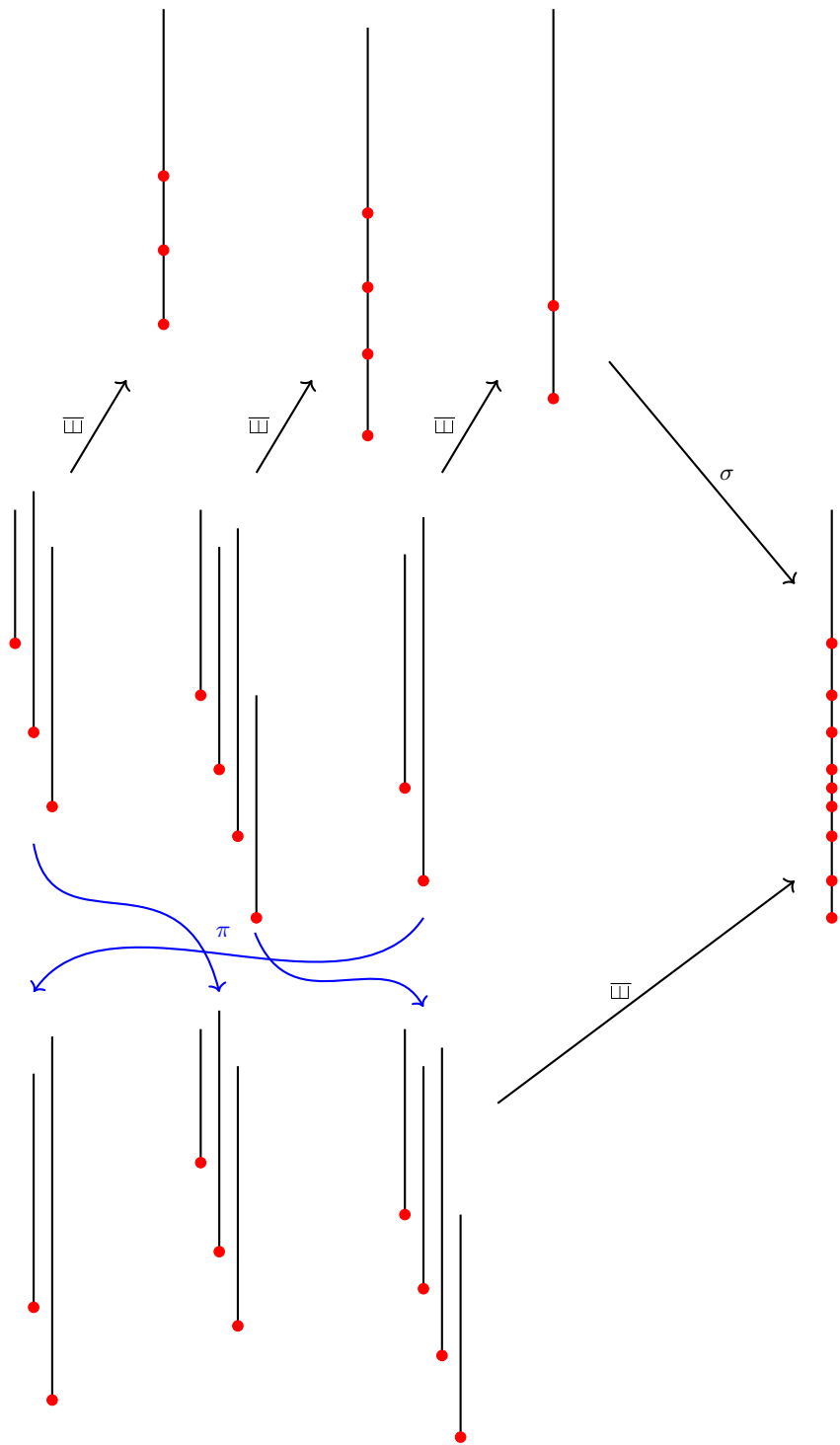
\begin{center}
	\begin{tikzpicture}
		\tikzset{
			position label/.style={
				below = 3pt,
				text height = 2ex,
				text depth = 1ex
			}
		}
		\tikzset{
			position label/.style={
				below = 3pt,
				text height = 1.5ex,
				text depth = 1ex
			},
			brace/.style={
				decoration={brace, mirror},
				decorate
			}
		}
		
		\tikzset{
			position label/.style={
				below = 3pt,
				text height = 1.5ex,
				text depth = 1ex
			},
			braceup/.style={
				decoration={brace},
				decorate
			}
		}

		\draw[black, thick] (-2,1.2) -- (-2,3);
		\filldraw[red] (-2,1.2) circle (2pt) node[anchor=north]{};
		\draw[black, thick] (-1.75,0) -- (-1.75,3.25);
		\filldraw[red] (-1.75,0) circle (2pt) node[anchor=north]{};
		\draw[black, thick] (-1.5,-1) -- (-1.5,2.5);
		\filldraw[red] (-1.5,-1) circle (2pt) node[anchor=north]{};
		
		
		\draw[black, thick] (0,9.75) -- (0,5.5) ;
		\filldraw[red] (0,5.5) circle (2pt) node[anchor=north]{};
		\filldraw[red] (0,6.5) circle (2pt) node[anchor=north]{};
		\filldraw[red] (0,7.5) circle (2pt) node[anchor=north]{};
		
		
		\draw[black, thick] (0.5,0.5) -- (0.5,3) ;
		\filldraw[red] (0.5,0.5) circle (2pt) node[anchor=north]{};
		\draw[black, thick] (0.75,-0.5) -- (0.75,2.5) ;
		\filldraw[red] (0.75,-0.5) circle (2pt) node[anchor=north]{};
		\draw[black, thick] (1,-1.4) -- (1,2.75) ;
		\filldraw[red] (1,-1.4) circle (2pt) node[anchor=north]{};
		\draw[black, thick] (1.25,-2.5) -- (1.25,0.5) ;
		\filldraw[red] (1.25,-2.5) circle (2pt) node[anchor=north]{};
		
		\draw[black, thick] (2.75,4) -- (2.75,9.5) ;
		\filldraw[red] (2.75,4) circle (2pt) node[anchor=north]{};
		\filldraw[red] (2.75,5.1) circle (2pt) node[anchor=north]{};
		\filldraw[red] (2.75,6) circle (2pt) node[anchor=north]{};
		\filldraw[red] (2.75,7) circle (2pt) node[anchor=north]{};
		
		\draw[black, thick] (3.25,2.4) -- (3.25,-0.75) ;
		\filldraw[red] (3.25,-0.75) circle (2pt) node[anchor=north]{};
		\draw[black, thick] (3.5,2.9) -- (3.5,-2) ;
		\filldraw[red] (3.5,-2) circle (2pt) node[anchor=north]{};
		
		\draw[black, thick] (5.25,4.5) -- (5.25,9.75) ;
		\filldraw[red] (5.25,4.5) circle (2pt) node[anchor=north]{};
		\filldraw[red] (5.25,5.75) circle (2pt) node[anchor=north]{};
		
		\draw[black, thick] (9,-2.5) -- (9,3) ;
		\filldraw[red] (9,1.2) circle (2pt) node[anchor=north]{};
		\filldraw[red] (9,0) circle (2pt) node[anchor=north]{};
		\filldraw[red] (9,-1) circle (2pt) node[anchor=north]{};
		\filldraw[red] (9,0.5) circle (2pt) node[anchor=north]{};
		\filldraw[red] (9,-0.5) circle (2pt) node[anchor=north]{};
		\filldraw[red] (9,-1.4) circle (2pt) node[anchor=north]{};
		\filldraw[red] (9,-2.5) circle (2pt) node[anchor=north]{};
		\filldraw[red] (9,-0.75) circle (2pt) node[anchor=north]{};
		\filldraw[red] (9,-2) circle (2pt) node[anchor=north]{};

		\draw[black, thick] (-1.75,-4.6) -- (-1.75,-7.75) ;
		\filldraw[red] (-1.75,-7.75) circle (2pt) node[anchor=north]{};
		\draw[black, thick] (-1.5,-4.1) -- (-1.5,-9) ;
		\filldraw[red] (-1.5,-9) circle (2pt) node[anchor=north]{};
		
		\draw[black, thick] (0.5,-5.8) -- (0.5,-4);
		\filldraw[red] (0.5,-5.8) circle (2pt) node[anchor=north]{};
		\draw[black, thick] (0.75,-7) -- (0.75,-3.75);
		\filldraw[red] (0.75,-7) circle (2pt) node[anchor=north]{};
		\draw[black, thick] (1,-8) -- (1,-4.5);
		\filldraw[red] (1,-8) circle (2pt) node[anchor=north]{};
		
		\draw[black, thick] (3.25,-6.5) -- (3.25,-4) ;
		\filldraw[red] (3.25,-6.5) circle (2pt) node[anchor=north]{};
		\draw[black, thick] (3.5,-7.5) -- (3.5,-4.5) ;
		\filldraw[red] (3.5,-7.5) circle (2pt) node[anchor=north]{};
		\draw[black, thick] (3.75,-8.4) -- (3.75,-4.25) ;
		\filldraw[red] (3.75,-8.4) circle (2pt) node[anchor=north]{};
		\draw[black, thick] (4,-9.5) -- (4,-6.5) ;
		\filldraw[red] (4,-9.5) circle (2pt) node[anchor=north]{};

		\draw [->,black,thick] (-1.25,3.5) -- (-0.5,4.75) node [position label, pos=0.5 , yshift=2.5ex, xshift=-2.5ex] {$\overline{\shuffle}$};
		
		\draw [->,black,thick] (1.25,3.5) -- (2,4.75) node [position label, pos=0.5 , yshift=2.5ex, xshift=-2.5ex] {$\overline{\shuffle}$};
		
		\draw [->,black,thick] (3.75,3.5) -- (4.5,4.75) node [position label, pos=0.5 , yshift=2.5ex, xshift=-2.5ex] {$\overline{\shuffle}$};
		
		\draw [->,black,thick] (6,5) -- (8.5,2) node [position label, pos=0.5 , yshift=2.5ex, xshift=2.5ex] {$\sigma$};
		
		\draw [->,black,thick] (4.5,-5) -- (8.5,-2) node [position label, pos=0.5 , yshift=2.5ex, xshift=-2.5ex] {$\overline{\shuffle}$};
		
		\draw[->, thick,blue]
		(-1.75,-1.5)
		.. controls (-1.5,-3)  and (0.25,-1.5) ..
		(0.75,-3.5) ;
		
		\draw[->, thick,blue]
		(3.5,-2.5)
		.. controls (2.5,-4)  and (-0.75,-2) ..
		(-1.75,-3.5)  node [position label, pos=0.5 , yshift=5ex, xshift=-0.5ex] {$\pi$}; ;
		
		\draw[->, thick,blue]
		(1.23,-2.7)
		.. controls (1.73,-4)  and (3,-2.8) ..
		(3.5,-3.7) ;
		
	\end{tikzpicture}
\end{center}

\begin{proof}[Proof of \autoref{lem:shsh}]
	Let $N_l \coloneqq |\boldk^{n^{l-1} + 1}| + \ldots + |\boldk^{n^l}|$ for $l = 1,\ldots,m$, and $N^l \coloneqq N_1 + \ldots + N_l$ for $l = 1,\ldots, m$, and $N \coloneqq N^m$. We have
	\begin{align*}
		&\sh(\osh(\boldk^1,\ldots,\boldk^{n_1}), \ldots, \osh(\boldk^{n^{m-1}+1},\ldots,\boldk^{n})) \\
		=&\{\!\{ \sigma^{-1}_*(\boldh^1\ldots\boldh^m) \mid \sigma \in \sh(N_1,\ldots,N_m);\\ &\boldh^l \in \osh(\boldk^{n^{l-1}+1},\ldots,\boldk^{n^l}), \ l = 1,\ldots,m\}\!\} \\
		=&\{\!\{ \sigma^{-1}_*(\boldh^1\ldots\boldh^m) \mid \sigma \in \sh(N_1,\ldots,N_m);\\
		&\boldh^l = \rho_{l*}^{-1}(\boldk^{n^{l-1}+1} \ldots \boldk^{n^l}); \\
		&\rho_l \in \osh(|\boldk^{n^{l-1} + 1}|, \ldots, |\boldk^{n^l}|), \ l = 1,\ldots,m\}\!\}.
	\end{align*}
	Now, denoting $(\rho_{1}, \ldots,\rho_{m})$ the element of $\mathfrak S_N$ which acts on $\{N^{l-1}+1,\ldots,N^l\}$ with $\rho_l$, we continue the chain of identities
	\begin{align*}
		=&\{\!\{ \sigma^{-1}_*(\rho_{1}, \ldots,\rho_{m})^{-1}_* (\boldk^1\ldots\boldk^{n_m}) \mid \sigma \in \sh(N_1,\ldots,N_m);\\ 
		&\rho_l \in \osh(|\boldk^{n^{l-1}+1}|,\ldots,|\boldk^{n^l}|), \ l = 1,\ldots,m\}\!\}\\
		=&\{\!\{ (\sigma \circ (\rho_{1}, \ldots,\rho_{m}))^{-1}_* (\boldk^1\ldots\boldk^{n_m}) \mid \sigma \in \sh(N_1,\ldots,N_m);\\ 
		&\rho_l \in \osh(|\boldk^{n^{l-1}+1}|,\ldots,|\boldk^{n^l}|), \ l = 1,\ldots,m\}\!\}
	\end{align*}
	since $(\rho_{1}^{-1}, \ldots,\rho_{m}^{-1}) = (\rho_{1}, \ldots,\rho_{m})^{-1}$ and thanks to the composition rule \eqref{eq:compRule}. Let $\pi$ denote the restriction of $\sigma \circ (\rho_{1}, \ldots,\rho_{m})$ to the set 
	\[
	T \coloneqq \{t_1,\ldots,t_n\}, \quad t_l \coloneqq |\boldk^1| +\ldots+ |\boldk^l|.
	\]
	Since the $\boldk^h$'s are all non-empty, $T$ is a subset of $\{1,\ldots,N\}$ of cardinality $n$, so after renumbering it we can consider $\pi$ as an element of $\mathfrak S_n$. Now, since $\rho_l$ is an ordered shuffle, it preserves the ordering of $\{|\boldk^{n^{l-1}+1}|,\ldots,|\boldk^{n^{l-1}+1}|+ \ldots+|\boldk^{n^l}|\}$, and since $\sigma$ is a shuffle it preserves the ordering $\{\rho_l(|\boldk^{n^{l-1}+1}|),\ldots,\rho_l(|\boldk^{n^l}|)\} \subseteq \{N^{l-1}+1,\ldots,N^l\}$. These two facts imply $\pi \in \sh(n_1,\ldots,n_m)$, and we have
	\begin{align*}
		&\mathrel{\phantom{=}}\{ \sigma \circ (\rho_{1}, \ldots,\rho_{m}) \mid \rho_l \in \osh(|\boldk^{n^{l-1}+1}|,\ldots,|\boldk^{n^l}|), \ l = 1,\ldots,m\} \\
		&= \{ \tau \in \sh(|\boldk^1|,\ldots,|\boldk^n|) \mid \tau(t_{\pi^{-1}(1)}) < \ldots < \tau(t_{\pi^{-1}(n)}) \}\\
		&\eqqcolon \overline{\sh^\pi}(|\boldk^1|,\ldots,|\boldk^n|)
	\end{align*}
	since any $ \tau \in \sh(|\boldk^1|,\ldots,|\boldk^n|)$ with $\tau(t_{\pi^{-1}(1)}) < \ldots < \tau(t_{\pi^{-1}(n)})$ for some $\pi \in \mathfrak S_n$ factors uniquely as $ \sigma \circ (\rho_{1}, \ldots,\rho_{m})$ with $\sigma$ acting on $T$ with $\pi$: this is evident from the fact that each $\rho_l$ acts on the segment $[N^{l-1}+1,N^l]$ and $\sigma$ acts on the whole segment $[1,N]$ but without altering the order in each $[N^{l-1}+1,N^l]$. This implies
	\begin{align*}
		&\{ \sigma \circ (\rho_{1}, \ldots,\rho_{m}) \mid \sigma \in \sh(N_1,\ldots,N_m);\\ &\rho_l \in \osh(|\boldk^{n^{l-1}+1}|,\ldots,|\boldk^{n^l}|), \ l = 1,\ldots,m\} \\
		= &\bigsqcup_{\pi \in \sh(n_1,\ldots,n_m)} \overline{\sh^\pi}(|\boldk^1|,\ldots,|\boldk^n|)
	\end{align*}
	because as $\sigma$ ranges over $\sh(N_1,\ldots,N_m)$ all $\sh(n_1,\ldots,n_m) \ni \pi$'s are obtained, and the $\overline{\sh^\pi}$'s are mutually disjoint since the ordered shuffle relations imposed by different $\pi$'s are mutually exclusive. Since 
	\[
	\osh(\boldk^{\pi^{-1}(1)},\ldots,\boldk^{\pi^{-1}(n)})  = \{\!\{ \tau^{-1}_* (\boldk^1\ldots\boldk^n) \mid \tau \in \overline{\sh^\pi}(|\boldk^1|,\ldots,|\boldk^n|) \}\!\}
	\]
	the proof is concluded. 
\end{proof}
\begin{expl}
	We illustrate the idea behind this lemma with an example. Let 
	\begin{equation}\label{eq:numbers}
		n_1 = 2, \ n_2 = 1; \quad |\boldk^1| = 3, \ |\boldk^2| = 2, \ |\boldk^3| = 3
	\end{equation}
	and with the notations of the proof let
	\begin{align*}
		&\rho_1 = \begin{pmatrix}
			1 & 2 & 3 & 4 & 5 \\ 1 & 3 & 4 & 2 & 5
		\end{pmatrix} \in \osh(3,2), \ \rho_2 = \begin{pmatrix} 1 & 2 & 3 \\ 1 & 2 & 3 \end{pmatrix} \in \osh(3), \\ \Rightarrow \ (\rho_1, \rho_2) &= \begin{pmatrix}
			1 & 2 & 3 & 4 & 5 & 6 & 7 & 8 \\ 1 & 3 & 4 & 2 & 5 & 6 & 7 & 8 
		\end{pmatrix} \\
		\sigma &= \begin{pmatrix}1 & 2 & 3 & 4 & 5 & 6 & 7 & 8 \\ 2 & 3 & 6 & 7 & 8 & 1 & 4 & 5 \end{pmatrix} \in \sh(5,3) \\
		\Rightarrow \ \tau \coloneqq \sigma \circ (\rho_1,\rho_2) &= \begin{pmatrix}1 & 2 & 3 & 4 & 5 & 6 & 7 & 8 \\ 2 & 6 & 7 & 3 & 8 & 1 & 4 & 5 \end{pmatrix} \in \overline{\sh^\pi}(3,2,3) \\
		\text{with } \pi &= \begin{pmatrix} 1 & 2 & 3 \\ 2 & 3 & 1 \end{pmatrix} \in \sh(2,1)
	\end{align*}
	because the restriction of $\sigma \circ (\rho_1, \rho_2)$ to $\{|\boldk^1|, |\boldk^1 + \boldk^2|, |\boldk^1 + \boldk^2 + \boldk^3|\}$ is $\big( \!\begin{smallmatrix}
		3 & 5 & 8 \\ 7 & 8 & 5
	\end{smallmatrix}\!\big)$, which coincides with $\pi$ after renumbering domain and codomain. Also note how, given $\tau$ and the numbers \eqref{eq:numbers} one can recover $\sigma, \rho_1,\rho_2$: first obtain $\rho_1,\rho_2$ considering how $\tau$ orders the segments $[1,5]$, $[6,8]$ (and renumbering) and $\sigma = \tau \circ (\rho_{1}, \ldots,\rho_{m})$. We therefore have, writing terminal elements in red
	\begin{align*}
		(\sigma \circ (\rho_1,\rho_2))^{-1}_* (\boldk_1 \boldk_2 \boldk_3) &= (k^3_1,k^1_1, k^2_1, k^3_2, {\color{red}k^3_3}, k^1_2, {\color{red}k^1_3}, {\color{red}k^2_2}) \\
		&\in \sh(\osh(\boldk^1,\boldk^2),\osh(\boldk^3)) \\
		&\in \osh(\boldk^3,\boldk^1,\boldk^2) = \osh(\boldk^{\pi^{-1}(1)},\boldk^{\pi^{-1}(2)},\boldk^{\pi^{-1}(3)})
	\end{align*}
	which can also be written as
	\begin{align*}
		(k^3_1,k^1_1, k^2_1, k^3_2, {\color{red}k^3_3}, k^1_2, {\color{red}k^1_3}, {\color{red}k^2_2}) &= \eta^{-1}_*(\boldk^3, \boldk^1, \boldk^2) \in \osh(\boldk^3,\boldk^1,\boldk^2) \\ \text{with }
		\eta &= \begin{pmatrix} 1 & 2 & 3 & 4 & 5 & 6 & 7 & 8 \\ 1 & 4 & 5 & 2 & 6 & 7 & 3 & 8  \end{pmatrix} \in \osh(3,3,2).
	\end{align*}
\end{expl}

We primarily use \autoref{lem:shsh} in the two cases $n_1 = \ldots = n_m = 1$ with $m$ arbitrary, and $m = 2$ with $n_1,n_2$ arbitrary; the former admits the following concise reformulation. Given a vector space $W$, let $\odot$ denote symmetric tensor product, and
\begin{equation}
	\odot_m \coloneqq \frac{1}{m!} \sum_{\pi \in \mathfrak S_m} \pi_* \colon W^{\otimes m} \twoheadrightarrow W^{\odot m},\quad w_1 \otimes \cdots \otimes w_m \mapsto w_1 \odot \cdots \odot w_m
\end{equation}
denote the symmetrisation map. When referring to the external tensor product we will replace the symbol $\odot$ with $\boxdot$.
\begin{cor}\label{thm:osym}
	The diagram
	\begin{equation}\label{diag:osym}
		\begin{tikzcd}[column sep = large]
			& T(V)^{\boxtimes m} \arrow[dd,"m!\boxdot_m"] \\
			T(V) \arrow[ru,"\wDOshuffle^m"]\arrow[rd,swap,"\wDshuffle^m"] \\
			& T(V)^{\boxtimes m}
		\end{tikzcd}
	\end{equation}
	commutes.
	\begin{proof}
		The statement in coordinates reads
		\begin{equation*}
			\sum_{\substack{\pi \in \mathfrak S_m \\ \boldh \in \osh(\boldk^{\pi(1)}, \ldots, \boldk^{\pi(m)})}} a^\boldh = \sum_{\boldk \in \sh(\boldk^1,\ldots,\boldk^m)} a^\boldk 
		\end{equation*}
		for non-empty tuples $\boldk^1,\ldots,\boldk^m$. Indeed, we have
		\begin{equation*}
			\begin{split}
				(m! \boxdot \wDOshuffle a)^{\boldk^1,\ldots,\boldk^m} &= \Big( \sum_{\substack{\pi \in \mathfrak S_m \\ (\widetilde a)^m_{\oshuffle}}} \pi_* (a_{(1)} \boxtimes \cdots \boxtimes a_{(m)}) \Big)^{\boldk^1,\ldots,\boldk^m} \\
				&= \sum_{\pi \in \mathfrak S_m} \sum_{(\widetilde a)^m_{\oshuffle}} a_{(\pi(1))}^{\boldk^1} \cdots a_{(\pi(m))}^{\boldk^m} \\
				&= \sum_{\pi \in \mathfrak S_m} \sum_{(\widetilde a)^m_{\oshuffle}} a_{(1)}^{\boldk^{\pi(1)}} \cdots a_{(m)}^{\boldk^{\pi(m)}} \\
				&=  \sum_{\pi \in \mathfrak S_m} (\wDOshuffle^m a)^{\boldk^{\pi(1)}, \ldots, \boldk^{\pi(m)}} \\
				&= \sum_{\substack{\pi \in \mathfrak S_m \\ \boldh \in \overline{\sh}(\boldk^{\pi(1)}, \ldots, \boldk^{\pi(m)})}} a^\boldh.
			\end{split}
		\end{equation*}
		To prove the claim in coordinates we must show the identity of sets
		\begin{equation*}
			\bigsqcup_{\pi \in \mathfrak S_m} \osh(\boldk^{\pi(1)}, \ldots, \boldk^{\pi(m)}) = \sh(\boldk^1,\ldots,\boldk^m)
		\end{equation*}
		for tuples $\boldk^1,\ldots,\boldk^m$ of positive order. This is precisely \autoref{lem:shsh} with $n_1 = \ldots = n_m = 1$, since $\sh(n_1,\ldots,n_m) = \mathfrak S_m$ and we may replace $\pi^{-1}$ with $\pi$.
	\end{proof}
\end{cor}

\section{Weakly geometric rough paths}\label{sec:grps}

We denote $T^N(V)$ the vector subspace of $T(V)$ given by all tensors of degree $\leq N$, and super/subscripts of $\leq N$, $\geq M$ denote truncations of the algebra to tensors of the degrees expressed (e.g.\ for $a \in T(V)$ to belong to $T^N(V)$ it means that $a^n = 0$ for $n > N$, or equivalently $a = a^{\leq N}$). We will similarly use $[M,N]$ as a super/subscript to denote tensors of degrees $n$ with $M \leq n \leq N$. A \emph{control} on $[0,T]$ is a continuous function $\omega$ defined on the simplex $\Delta_T \coloneqq \{(s,t) \in [0,T]^2 \mid s \leq t\}$, s.t.\ $\omega(t,t) = 0$ for $0 \leq t \leq T$ and is superadditive, i.e.\ $\omega(s,u) + \omega(u,t) \leq \omega(s,t)$ for $0 \leq s \leq u \leq t \leq T$.
\begin{defn}[Weakly geometric rough path]\label{def:rp}
	Let $T > 0$, $p \geq 1$ and $\omega$ be a control on $[0,T]$. A $p$-\emph{weakly geometric rough path} $\bfX$ controlled by $\omega$, defined on $[0,T]$ and with values in $V$ may be defined as a continuous map 
	\begin{equation}
		\bfX \colon \Delta_T \to T^{\lfloor p \rfloor}(V)
	\end{equation}
	with $\bfX^0 = 1$ and satisfying the following properties, which we first present in coordinate-free form and subsequently in coordinates w.r.t.\ a basis of $V$:
	\begin{description}
		\item[Regularity.] $\displaystyle \sup_{0 \leq s < t \leq T}\frac{\lvert \bfX_{st}^n\rvert}{\omega(s,t)^{n/p}} < \infty$, or $\displaystyle \sup_{0 \leq s < t \leq T}\frac{\lvert \bfX_{st}^{\boldk}\rvert}{\omega(s,t)^{|\boldk|/p}} < \infty$ for $n = |\boldk| = 1,\ldots \p$;
		\item[Multiplicativity.] $(\bfX_{su} \otimes \bfX_{ut})^{\leq \p} = \bfX_{st}$, i.e.\ $ \bfX_{st}^{\boldk} = \displaystyle \sum_{(\boldi, \boldj) = \boldk} \bfX_{su}^{\boldi} \bfX_{ut}^{\boldj}$ for $|\boldk| \leq \p$ and $0 \leq s \leq u \leq t \leq T$;
		\item[Integration by parts.] $\bfX_{st} \boxtimes \bfX_{st} = \Dshuffle \bfX_{st}$, or $\displaystyle \bfX_{st}^\boldi \bfX_{st}^\boldj = \sum_{\boldk \in \sh(\boldi, \boldj)} \bfX_{st}^\boldk$ for all $0 \leq s \leq t \leq T$.
	\end{description}
	Let $\mathscr C^p_\omega([0,T],V)$ denote the set of all such maps.
\end{defn}
In the following we will refer to such objects as \say{rough paths}, dropping the \say{weakly geometric}, since these are the only rough paths that we will be considering in this paper. We will sometimes refer to the third property above as geometricity, since it distinguishes weakly geometric rough paths among the more general branched rough paths, which we do not treat here. We will always denote rough paths in bold. The last condition is usually stated by saying that $\bfX$ takes values in the group $G^{\p}(V)$, defined in \cite[Definition 2.9]{CDLL16}. We will denote $X \coloneqq \bfX^1$ the \emph{trace} of X: when equipped with an initial value $X_0$ (which will often be provided) this is an element of $\mathcal C_\omega^p ([0, T ], V )$ defined as the set of continuous paths $Y \colon [0, T ] \to V$ s.t., denoting $Y_{st} \coloneqq Y_t - Y_s$ (a notation that will be used for paths throughout)
\begin{equation}
	\sup_{0 \leq s < t \leq T} \frac{|Y_{st}|}{\omega(s,t)^{1/p}}< \infty.
\end{equation}
It is sufficient to define $\bfX$ to take values in $T^{\p}(V)$, as \cite[Theorem 2.2.1]{Lyo98} shows that there exists a unique extension of $\widehat{\bfX}$ of $\bfX$ to $T(V)$ which satisfies the above three properties, and for $m >\p$ is given by
\begin{equation}\label{eq:lyonsLim}
	\widehat{\bfX}^m_{st} = \lim_{n \to \infty} \Big(\bigotimes_{[u,v] \in \pi_n} \bfX_{uv}\Big)^m
\end{equation}
where $(\pi_n)_n$ is any sequence of partitions on $[s,t]$ with vanishing step size as $n \to \infty$.

The following proposition states that the symmetric part of a weakly geometric rough path is entirely determined by its trace. Given $\ell \in \mathcal L(T^N(V),U)$ ($\mathcal L$ denotes the space of linear maps) and $a \in T(V)$ we will denote $\langle \ell, a \rangle = \ell(a)$ the evaluation of $\ell$ on $a$. We will always identify $\mathcal L(\mathbb R, U) = U$ by setting $\ell \mapsto \ell(1)$.
\begin{prop}\label{prop:sym}
	For $\bfX \in \mathscr C^p_\omega([0,T],V)$ we have $n!\odot_n \bfX^n_{st} = X^{\otimes n}_{st}$.
	\begin{proof}
		We proceed by induction on $n$. For $n = 0,1$ there is nothing to prove. For the inductive step we will need the following fact: each $\pi \in \mathfrak S_{n+1}$ can be expressed uniquely as $\sigma \circ \rho$ with $\rho$ in the stabiliser of $n+1$ (a subgroup of $\mathfrak S_{n+1}$ isomorphic to $\mathfrak S_n$) and $\sigma \in \sh(n,1)$: indeed, if $m \coloneqq \pi(n+1)$ we may set
		\[
		\sigma(k) \coloneqq \begin{cases} k &1 \leq k \leq m-1 \\ k+1 & m \leq k \leq n \\ m &k = n+1
		\end{cases} \quad
		\rho(k) \coloneqq \begin{cases} \pi(k) &1 \leq \pi(k) \leq m-1 \\ \pi(k)-1 &m \leq \pi(k) \leq n\\ n+1 &k = n+1
		\end{cases}.
		\]
		Uniqueness follows from a counting argument, since there are $n!$ choices for $\rho$ and $n+1$ for $\sigma$. We then compute
		\begin{align*}
			(n+1)! \odot_{n+1} \bfX_{st}^{n+1} &= \sum_{\pi \in \mathfrak S_{n+1}} \pi_* \bfX^{n+1}_{st} \\
			&= \sum_{\rho \in \mathfrak S_n} \rho_* \sum_{\sigma \in \sh(n,1)} \sigma_* \bfX_{st}^{n+1} \\
			&=  \sum_{\rho \in \mathfrak S_n} \rho_*( \Dshuffle\bfX_{st})^{n,1} \\
			&= \sum_{\rho \in \mathfrak S_n} \rho_*\bfX_{st}^n \otimes X_{st} \\
			&= X_{st}^{\otimes (n+1)}
		\end{align*}
		where we have used the geometricity axiom and the inductive hypothesis.
	\end{proof}
\end{prop}

We proceed to define the objects that can be regarded as dual to rough paths, original to \cite{Gub04}. In what follows we will write $\approx_m$ between two real-valued quantities dependent on $0 \leq s \leq t \leq T$ to mean that their difference lies in $O(\omega(s,t)^{m/p})$ as $t \searrow s$, and simply $\approx$ to mean $\approx_{\p + 1}$. We will frequently use the following properties, which are trivial to check:
\begin{equation}\label{eq:approxProp}
	\begin{split}
		a_{st} \approx_m b_{st} \approx_n c_{st} &\Rightarrow a_{st} \approx_{n\wedge m} c_{st}\\
		a_{st} \approx_m b_{st}, \ c_{st} \approx_n 0 &\Rightarrow a_{st}  c_{st} \approx_{m+n} b_{st}  c_{st}
	\end{split}
\end{equation}
from which we deduce more generally
\begin{equation}\label{eq:generalApprox}
	\begin{split}
		&\mathrel{\phantom{\Rightarrow}}a_{st} \approx_{m_1} b_{st}, \ a_{st},b_{st} \approx_{n_1} 0, \quad c_{st} \approx_{m_2} d_{st},\ c_{st},d_{st} \approx_{n_2} 0 \\
		&\Rightarrow a_{st}  c_{st} \approx_{m_1 + n_2} b_{st}  c_{st} \approx_{m_2+n_1} b_{st}  d_{st} \\
		&\Rightarrow a_{st}  c_{st} \approx_{(m_1+n_2) \wedge (m_2+n_1)} b_{st}  d_{st}.
	\end{split}
\end{equation}
If a continuous map $\widetilde\bfX \colon \Delta_T \to T^{\lfloor p \rfloor}(V)$ satisfies the regularity and integration by parts conditions, and satisfies the multiplicativity condition with a \say{$\approx$} replacing the \say{$=$} (\emph{almost multiplicative}), it defines a rough path by \cite[Theorem 3.3.1]{Lyo98}, by taking the limit \eqref{eq:lyonsLim} (w.r.t.\ $\widetilde \bfX$), and this rough path $\bfX$ is unique with the property that $\widetilde \bfX_{st} \approx \bfX_{st}$. The following lemma tells us that this is also true if the integration by parts condition only holds with an $\approx$ (\emph{almost geometric}). In light of this, we will break with the literature in defining an \emph{almost rough path} as an $\widetilde\bfX$ that satisfies the regularity condition in \autoref{def:rp} and is almost multiplicative and almost geometric.
\begin{prop}[Almost rough paths]\label{prop:geomAlm}
	Let $\widetilde \bfX$ be a $V$-valued almost $p$-rough path. Then there exists a unique $p$-rough path $\bfX$ with the property that $\bfX_{st} \approx \widetilde \bfX_{st}$.
	\begin{proof}
		We use that the shuffle algebra is free abelian over the Lyndon words \cite[Theorem 6.1]{Reu93} to define an intermediate $\overline \bfX$: set $\overline\bfX{}^\boldh \coloneqq \widetilde\bfX{}^\boldh$ if $\boldh$ is a Lyndon word with $|\boldh| \leq \p$, and for a tuple $\boldk$ with $|\boldk| < \p$ expressed (uniquely up to order of factors) as $\sum_\lambda c_\lambda \boldk^1_\lambda \shuffle \cdots \shuffle \boldk^{n_\lambda}_\lambda$ with the $\boldk^j_\lambda$'s (not necessarily distinct) Lyndon words set
		\[
		\overline \bfX{}^{\boldk} \coloneqq \sum_\lambda c_\lambda \overline \bfX{}^{\boldk^1_\lambda} \cdots \overline \bfX{}^{\boldk^{n_{\scaleto{\lambda}{2.7pt}}}_\lambda} =  \sum_\lambda c_\lambda \widetilde \bfX{}^{\boldk^1_\lambda} \cdots \widetilde \bfX{}^{\boldk^{n_{\scaleto{\lambda}{2.7pt}}}_\lambda} \approx \sum_{\lambda} c_\lambda \langle \boldk^1_\lambda\shuffle \cdots \shuffle\boldk^{n_\lambda}_\lambda , \widetilde \bfX \rangle = 
		\widetilde\bfX{}^\boldk
		\]
		since $\widetilde\bfX$ is almost geometric. $\overline \bfX$ is then $\approx \widetilde\bfX$, it satisfies integration by parts (exactly) by construction since if $\boldi = \sum_\mu a_\mu \boldi_\mu^1 \shuffle \cdots \shuffle \boldi^{m_\mu}_\mu$, $\boldj = \sum_\nu b_\nu \boldj_\nu^1 \shuffle \cdots \shuffle \boldj^{n_\nu}_\nu$ with $a_\mu, b_\nu \in \mathbb Q$ and $\boldi^q_\mu, \boldj^r_\nu$ Lyndon
		\[
		\overline \bfX{}^{\boldi} \overline \bfX{}^{\boldj} = \sum_{\mu, \nu} a_\mu b_\nu \overline \bfX{}^{\boldi_\mu^{1}} \cdots \overline \bfX{}^{\boldi^{m_{\scaleto{\mu}{2.7pt}}}_\mu} \overline\bfX{}^{\boldj_\nu^1} \cdots \bfX^{\boldj^{n_{\scaleto{\nu}{1.9pt}}}_\nu} =  \langle \sum_{\mu, \nu} a_\mu b_\nu \boldi_\mu^{1} \shuffle \cdots \shuffle \boldj^{m_\mu}_\mu \shuffle \boldj_\nu^{1} \shuffle \cdots \shuffle \boldj^{n_\nu}_\nu, \overline \bfX\rangle = \langle \boldi \shuffle \boldj , \overline \bfX \rangle
		\]
		and is still almost multiplicative since 
		\[
		\overline \bfX{}^{\boldk} \approx  \bfX^{\boldk}_{st} = \sum_{(\boldi, \boldj) = \boldk} \bfX_{su}^{\boldi} \bfX_{ut}^{\boldj} \approx \sum_{(\boldi, \boldj) = \boldk} \overline\bfX{}_{su}^{\boldi} \overline\bfX{}_{ut}^{\boldj}
		\]
		Existence of $\bfX$ then follows immediately by applying the above-referenced result to $\overline\bfX$ and uniqueness follows from the fact that if $\bfX'$ is a second $p$-rough path satisfying the statement of this proposition, then we have $\bfX' \approx \widetilde \bfX \approx \overline \bfX \ \Rightarrow \ \bfX' = \bfX$ again by the same result.
	\end{proof}
\end{prop}

\begin{defn}
	Let $\bfX$ be as above and $U$ another vector space. An $U$-valued $\bfX$-\emph{controlled path} $\bfH$ is an element of $\mathcal C^p_\omega([0,T],\mathcal L(T^{\p - 1}(V),U))$ (where $\omega$ is the control for $\bfX$) s.t.\ for $n = 0,\ldots, \lfloor p \rfloor -2$ and each $a \in V^{\otimes n}$
	\begin{equation}\label{eq:controlled}
		\langle \bfH_{n;t}, a \rangle \approx_{\p - n} \langle \bfH_{[n,\p-1];s}, \bfX_{st}^{\leq\p-1-n} \otimes a \rangle.
	\end{equation}
	Here $\bfH_n$ denotes the $n$-th level of $\bfH$. Denote $\mathscr D_{\bfX}(U)$ the vector space of all $U$-valued $\bfX$-controlled paths.
\end{defn}
We will always denote controlled paths with an overline. The maps $\bfH_n \coloneqq \bfH|_{V^{\otimes n}}$ are known as the \emph{Gubinelli derivatives} of $H$ and $H \coloneqq \bfH_0 \in U$ is called the \emph{trace} of $\bfH$ (note the discrepancy with rough paths: for these the trace is the order-1 component). Note that the defining condition only involves $\bfX^{\leq \lfloor p \rfloor - 1}$, and that it holds automatically at level $\p - 1$ by regularity of $\bfH$. In coordinates it reads
\begin{equation}\label{eq:contrDefCoords}
	\bfH_{\boldb;t}^{k} \approx_{\p - |\boldb|} \sum_{|\bolda| = 0}^{\p - 1 - |\boldb|} \bfH^k_{(\bolda,\boldb);s} \bfX^{\bolda}_{st}, \quad 0 \leq |\boldb| \leq \p - 2.
\end{equation}
Here the superscript $k$ refers to the value of $\bfH$ in $U$ (and will often be omitted when unnecessary), and the sum is not only on the length $|\bolda|$ of the tuple $\bolda$, but on the tuple itself. For the branched version of this definition, see \cite{HK15}.

An important case is when $U = \mathcal L(V,W)$ for another vector space $W$: by the tensor-hom adjunction we then have
\begin{equation}\label{eq:adj}
	\mathcal L(T^{\lfloor p \rfloor - 1}(V),\mathcal L(V,W)) = \mathcal L(T^{\lfloor p \rfloor - 1}(V) \otimes V,W) = \mathcal L \Big(\bigoplus_{n = 1}^\p V^{\otimes n},W \Big).
\end{equation}
We will use angle brackets and coordinate notation for linear maps accordingly, i.e.\ the last slot in a bracket or in a tuple will refer to the copy of $V$ in the target space of the original linear map. We will call controlled paths valued in $\mathcal L(V,W)$ \emph{$W$-valued controlled integrands}, and we may rewrite \eqref{eq:controlled} as
\begin{equation}\label{eq:contrInt}
	\langle \bfH_t, b \rangle \approx_{\p - n +1} \langle \bfH_{[n,\p];s}, \bfX_{st}^{[0,\p - n]} \otimes b \rangle \in W, \quad b \in V^{\otimes n}, \quad n = 1,\ldots,\p-1
\end{equation}
or in coordinates
\begin{equation}\label{eq:contrIntCoords}
	\bfH^k_{\boldb;t} \approx_{\p - |\boldb|+1} \sum_{|\bolda| = 0}^{\p - |\boldb|} \bfH^k_{(\bolda,\boldb);s} \bfX^{\bolda}_{st}, \quad 1 \leq |\boldb| \leq \p - 1.
\end{equation}
The next example contains a very important example of controlled path.
\begin{expl}[Smooth functions of $X$]\label{ex:FX}
	Let $F \in C^\infty(V,U)$, then
	\begin{equation}
		t \mapsto (F(X_t), DF(X_t), \ldots, D^{\p-1} F(X_t)) \in \mathcal L(T^{\p -1}(V), U)
	\end{equation}
	is an $\bfX$-controlled path, which we denote $\overline{F}(X)$. Indeed, denoting by $\partial_{\boldc}F$ the order-$|\boldc|$ partial derivative of $F$ in the directions of the chosen basis determined by the tuple $\boldc$, we have, for $0 \leq |\boldb| \leq \p - 2$
	\begin{align*}
		&\mathrel{\phantom{=}}\overline{F}(X_t)_{\boldb} - \sum_{|\bolda| = 0}^{\p - 1 - |\boldb|} \overline{F}(X_s)_{(\bolda,\boldb)} \bfX^{\bolda}_{st}\\
		&= \partial_\boldb F(X_t) - \sum_{|\bolda| = 0}^{\p - 1 - |\boldb|} \partial_{\bolda,\boldb} F(X_s) \bfX^{\bolda}_{st} \\
		&= \partial_\boldb F(X_t) - \sum_{n = 0}^{\p - 1 - |\boldb|} \frac{1}{n!} \partial_{\bolda,\boldb} F(X_s) X_{st}^{\alpha_1} \cdots X_{st}^{\alpha_n} \\
		&\approx_{\p - |\boldb|} 0
	\end{align*}
	where we have used \autoref{prop:sym} together with the symmetry of higher differentials and Taylor's approximation. Note that the symmetry of Gubinelli derivatives is a special feature of this kind of controlled path, and cannot be expected to hold in general. When $U = \mathcal L(V,W)$ we shall call $F$ an \emph{$W$-valued 1-form}, and we adopt the convention
	\[
	\langle \overline F(X), v_1 \otimes \cdots \otimes v_{n+1} \rangle = D^nF(X)(v_1, \ldots, v_n)(v_{n+1}) \in W
	\]
	or in coordinates $\overline F(X)_{\bolda,\beta} = \partial_{\bolda} F_\beta(X)$.
\end{expl}
The next lemma is necessary for the definition of rough integral.
\begin{lem}
	Let $\bfX \in \mathscr C^p_\omega([0,T],V)$ and $\bfH \in \mathscr D_{\bfX}(\mathcal L(V,W))$. Define, for $0 \leq s \leq t \leq T$
	\begin{equation}
		\Xi_{st} \coloneqq \langle \bfH_s, \bfX_{st}^{\geq 1} \rangle \in W
	\end{equation}
	where the evaluation is taken under the identification \eqref{eq:adj}. Then $\Xi$ is almost additive: for all $0\leq s \leq u \leq t \leq T$
	\begin{equation}
		\Xi_{st} -  \Xi_{su} -  \Xi_{ut} \approx 0
	\end{equation}
	\begin{proof}
		Using the multiplicativity axiom, the regularity of $\bfH$ and \eqref{eq:contrIntCoords} (together with \eqref{eq:approxProp}) we may write
		\begin{align*}
			&\mathrel{\phantom{=}} \Xi_{st} -  \Xi_{su} -  \Xi_{ut} \\
			&= \sum_{|\boldc| = 1}^\p (\bfH_{\boldc;s} \bfX_{st}^\boldc  - \bfH_{\boldc;s} \bfX_{su}^\boldc  - \bfH_{\boldc;u} \bfX_{ut}^\boldc ) \\
			&= \sum_{|\boldc| = 1}^\p (\bfH_{\boldc;s} (\bfX_{st}^\boldc  - \bfX_{su}^\boldc - \bfX_{ut}^\boldc ) - \bfH_{\boldc;su} \bfX_{ut}^\boldc ) \\
			&= \sum_{|\boldc| = 1}^\p \Big(\bfH_{\boldc;s} \sum_{\substack{(\bolda,\boldb) = \boldc \\ |\bolda|, |\boldb| \geq 1}} \bfX_{su}^\bolda \bfX_{ut}^\boldb - \bfH_{\boldc;su} \bfX_{ut}^\boldc \Big) \\
			&\approx \sum_{\substack{|\bolda|, |\boldb| \geq 1 \\ |\bolda| + |\boldb| \leq \p}} \bfH_{(\bolda, \boldb);s} \bfX_{su}^\bolda \bfX_{ut}^\boldb - \sum_{|\boldsymbol \varepsilon| = 1}^{\p - 1} \bfH_{\boldsymbol{\varepsilon};su} \bfX_{ut}^{\boldsymbol{\varepsilon}} \\
			&\approx \sum_{\substack{|\bolda|, |\boldb| \geq 1 \\ |\bolda| + |\boldb| \leq \p}} \bfH_{(\bolda, \boldb);s} \bfX_{su}^\bolda \bfX_{ut}^\boldb - \sum_{|\boldsymbol \varepsilon| = 1}^{\p - 1} \sum_{ |\boldsymbol{\delta}| = 1}^{\p - |\boldsymbol \varepsilon|}\bfH_{(\boldsymbol{\delta},\boldsymbol{\varepsilon});s} \bfX_{su}^{\boldsymbol{\delta}} \bfX_{ut}^{\boldsymbol{\varepsilon}} \\
			&= 0
		\end{align*}
	\end{proof}
\end{lem}

\begin{defn}[Rough integral]\label{def:rint}
	Let $\bfX \in \mathscr C^p_\omega([0,T],V)$, $\bfH \in \mathscr D_{\bfX}(\mathcal L(V,W))$ be as above. We define
	\begin{equation}
		\int_0^\cdot \bfH \text{d} \bfX \colon [0,T] \to W
	\end{equation}
	to be the unique path $I \in \mathcal C_\omega([0,T],W)$ with the property that $I_{st} \approx \Xi_{st}$, which exists by \cite[Theorem 3.3.1]{Lyo98}, and is given by
	\[
	I_{st} = \lim_{n \to \infty}\sum_{[u,v] \in \pi_n} \Xi_{uv}	
	\]
	for a sequence of partitions $(\pi_n)_n$ with vanishing step size. We can make $\int \bfH \text{d} \bfX$ into an $\bfX$-controlled path by defining, for $1 \leq n \leq \p -1$
	\[
	\bigg(\overline \int \bfH \text{d} \bfX \bigg)_n \coloneqq \bfH_{n-1} \in \mathcal L(V^{\otimes n - 1},\mathcal L(V,W)) = \mathcal L(V^{\otimes n},W).
	\]
	Note the presence of the bar above the integral sign, which emphasises membership to $\mathscr D_{\bfX}(W)$.
\end{defn}

An $\bfX$-controlled path can be made into a rough path in its own right. We use \eqref{eq:liftMotivation} as a blueprint for the following definition, where we truncate at the correct order to avoid infinite sums.
\begin{defn}[Lift of a controlled path]\label{def:lift}
	Let $\bfX \in \mathscr C^p_\omega([0,T],V)$, $\bfH \in \mathscr D_{\bfX}(U)$. Define $\upharpoonleft_{\bfX} \!\! \bfH \colon \Delta_T \to T^{\p}(U)$ (notice the partial arrow notation to indicate almost multiplicativity \& geometricity) by
	\begin{equation}
		(\upharpoonleft_{\bfX} \!\! \bfH)^0_{st} \coloneqq 1,\quad (\upharpoonleft_{\bfX} \!\! \bfH)^1_{st} \coloneqq H_{st}
	\end{equation}
	and for $2\leq m \leq \p$
	\begin{equation}\label{eq:partialArrow}
		\begin{split}
			(\upharpoonleft_{\bfX} \!\! \bfH)^m_{st} &\coloneqq \langle \bfH^{\boxtimes m}_s, \wDOshuffle^m \bfX_{st} \rangle \\
			&=  \sum_{\substack{n_1, \ldots, n_m \geq 1 \\ n \coloneqq n_1 + \ldots + n_m \leq \p}} \langle \bfH_{n_1;s} \boxtimes \cdots \boxtimes \bfH_{n_m;s}, (\wDOshuffle \bfX)^{n_1,\ldots,n_m}_{st} \rangle.
		\end{split}
	\end{equation}
	As it is shown in \autoref{thm:almost} below, \cite[Theorem 3.3.1]{Lyo98} applies to this functional, and given any sequence of partitions $(\pi_n)_n$ with vanishing step size
	\begin{equation}\label{eq:truelift}
		(\xhtrue)_{st} \coloneqq \lim_{n \to \infty} \Big( \bigotimes_{[u,v] \in \pi_n} (\xhapprox)_{uv} \Big)
	\end{equation}
	defines an element of $\mathscr C^p_\omega([0,T],U)$, which we call the \emph{lift} of $\bfH$ to rough path w.r.t.\ $\bfX$.
\end{defn}

Note how it was necessary to distinguish the case $m = 1$ above: this is due to the fact that we do not have the $\p^\text{th}$ Gubinelli derivative, and therefore \eqref{eq:partialArrow} would only be accurate at order $\p$ (though in all explicit cases presented here these are known, and \eqref{eq:partialArrow} is applicable for $m = 1$ too; an example where the case distinction is essential would be \autoref{ex:FX} with $F$ only $(\p -1)$-times differentiable). \eqref{eq:partialArrow} can be written dually as
\begin{equation}
	(\xhapprox)^m_{st} = \langle \bfH_{\geq 1; s}^{\oshuffle m}, \bfX_{st} \rangle = \sum_{\substack{n_1, \ldots, n_m \geq 1 \\ n \coloneqq n_1 + \ldots + n_m \leq \p}} \langle \bfH_s^{n_1} \oshuffle \cdots \oshuffle \bfH_s^{n_m}, \bfX^n_{st} \rangle
\end{equation}
and in coordinates as
\begin{equation}\label{eq:lastSumIn}
	\begin{split}
		(\upharpoonleft_{\bfX} \!\! \bfH)^{()} &= 1, \quad (\upharpoonleft_{\bfX} \!\! \bfH)^k = H^k \\
		(\upharpoonleft_{\bfX} \!\! \bfH)^{(k_1,\ldots,k_m)}_{st} &= \sum_{\substack{|\boldc^1|, \ldots , |\boldc^m| \geq 1 \\ |\boldc^1|+ \ldots + |\boldc^m| \leq \p \\ \boldc \in \osh(\boldc^1,\ldots,\boldc^m)}} \bfH^{k_1}_{\boldc^1;s} \cdots \bfH^{k_m}_{\boldc^m;s} \bfX^\boldc_{st}
	\end{split}
\end{equation}
In explicit calculations we will use $\xhapprox$, for which we have a combinatorial expression, as a proxy for the true lift $\xhtrue$.

The next result is one of the main theorems in this article. It can be compared with \cite[Theorem 4.6]{LCL07}, which applies to the special case of integrals of $\text{Lip}(\gamma)$ forms, covered in \autoref{expl:roughIntLift} below. Their proof makes use of the symmetry of $\text{Lip}(\gamma)$ forms, while the lemma below does not require it. Since some of the indexing is quite complex, it will be helpful to denote
\begin{equation}
	\begin{split}
		A_1 &\coloneqq \bigcup_{m=1}^\infty (I^\bullet_*)^m  \\
		A^p_1 &\coloneqq \bigcup_{m=1}^\infty \big\{ (\boldc^1,\ldots,\boldc^m) \in (I^\bullet_*)^m \ \big\mid \ |\boldc^1| + \ldots + |\boldc^m| \leq \p \big\}
	\end{split}
\end{equation}
so for example the last sum in \eqref{eq:lastSumIn} can be written
\[
\sum_{\substack{(\boldc^1,\ldots,\boldc^m) \in A^p_1 \\ \boldc \in \osh(\boldc^1,\ldots,\boldc^m)}} \bfH^{k_1}_{\boldc^1;s} \cdots \bfH^{k_m}_{\boldc^m;s} \bfX^\boldc_{st}
\]
where $m$ is intended as fixed.
\begin{thm}\label{thm:almost}
	$\xhapprox$ is an almost rough path. Therefore the limit taken in \eqref{eq:truelift} exists and defines a $U$-valued $p$-weakly geometric rough path, controlled by $\omega$ on $[0,T]$, with trace $H$.
	\begin{proof}
		We begin by showing almost multiplicativity, i.e.\ that for $|\boldk| = 0,\ldots,\p$ and $0 \leq s \leq u \leq t \leq T$
		\[
		\sum_{(\boldi, \boldj) = \boldk} (\xhapprox)^\boldi_{su} (\xhapprox)^\boldj_{ut} \approx (\upharpoonleft_{\bfX} \!\! \bfH)^\boldk_{st}.
		\]
		For $|\boldk| = 0$ this is trivial and for $|\boldk| = 1$ it coincides with the statement that $H$ is a path. For $|\boldk| = 2$ (which presupposes $\p \geq 2$) we have
		\begin{equation}\label{eq:k2}
			\begin{split}
				&\mathrel{\phantom{=}} \sum_{(\boldi, \boldj) = (k_1,k_2)} (\xhapprox)^\boldi_{su} (\xhapprox)^\boldj_{ut} \\
				&= (\xhapprox)^{(k_1,k_2)}_{su} + (\xhapprox)^{k_1}_{su} (\xhapprox)^{k_2}_{ut} + (\xhapprox)^{(k_1,k_2)}_{ut} \\
				&= \!\!\!\! \sum_{\substack{(\bolda^1,\bolda^2) \in A^p_1 \\ \bolda \in \osh(\bolda^1,\bolda^2)}} \!\!\!\!(\bfH^{k_1}_{\bolda^1;s} \bfH^{k_2}_{\bolda^2;s} \bfX^\bolda_{su}) + H^{k_1}_{su} H^{k_2}_{ut} + \!\!\!\! \sum_{\substack{(\boldb^1,\boldb^2) \in A^p_1 \\ \boldb \in \osh (\boldb^1,\boldb^2)}} \!\!\!\!(\bfH^{k_1}_{\boldb^1;u} \bfH^{k_2}_{\boldb^2;u} \bfX^\boldb_{ut}).
			\end{split}
		\end{equation}
		We continue the calculation by re-expanding all the $\bfH$ terms at $s$ and using \eqref{eq:approxProp}:
		\[
		H^{k_1}_{su} \approx_\p \sum_{|\bolda^1| = 1}^{\p - 1} \bfH_{\bolda^1;s}^{k_1} \bfX^{\bolda^1}_{su}
		\]
		\[
		H^{k_2}_{ut} \approx_\p \sum_{|\boldb| = 1}^{\p - 1} \bfH_{\boldb;u} \bfX^{\boldb}_{ut} \approx_\p \sum_{\substack{|\boldb| = 1,\ldots, \p -1 \\ |\bolda^2| = 0, \ldots,\p - 1-|\boldb|}}\bfH^{k_2}_{(\bolda^2,\boldb);s} \bfX_{su}^{\bolda^2} \bfX^{\boldb}_{ut}.
		\]
		These two identities, the fact that $H_{su},H_{ut} \approx_1 0$ and \eqref{eq:generalApprox} imply
		\[
		H^{k_1}_{su}H^{k_2}_{ut} \approx \sum_{\substack{|\boldb| = 1,\ldots, \p -1 \\|\bolda^1| = 1,\ldots, \p -1 \\ |\bolda^2| = 0, \ldots,\p - 1-|\boldb| \\ \bolda \in \sh(\bolda^1,\bolda^2)}} \bfH^{k_1}_{\bolda^1;s} \bfH^{k_2}_{(\bolda^2,\boldb);s} \bfX^\bolda_{su}\bfX^{\boldb}_{ut}.
		\]
		Similarly
		\[
		\bfH^{k_1}_{\boldb^1;u} \bfH^{k_2}_{\boldb^2;u} \bfX^\boldb_{ut} \approx \sum_{\substack{|\bolda^1| = 0,\ldots,\p - 1 - |\boldb^1| \\ |\bolda^2| = 0,\ldots,\p - 1 - |\boldb^2| \\ \bolda \in \sh(\bolda^1,\bolda^2)}} \bfH^{k_1}_{(\bolda^1,\boldb^1);s} \bfH^{k_2}_{(\bolda^2,\boldb^2);s} \bfX^\bolda_{su} \bfX^\boldb_{ut}.
		\]
		Incorporating these computations in \eqref{eq:k2} we obtain
		\begin{align*}
			&\mathrel{\phantom{=}} \!\!\!\! \sum_{(\boldi, \boldj) = (k_1,k_2)} \!\!\!\! (\xhapprox)^\boldi_{su} (\xhapprox)^\boldj_{ut} \\
			&\approx \!\!\!\! \sum_{\substack{(\bolda^1,\bolda^2) \in A^p_1 \\ \bolda \in \osh(\bolda^1,\bolda^2)}} \!\!\!\! (\bfH^{k_1}_{\bolda^1;s} \bfH^{k_2}_{\bolda^2;s} \bfX^\bolda_{su}) + \!\!\!\! \sum_{\substack{|\boldb| = 1,\ldots, \p -1 \\|\bolda^1| = 1,\ldots, \p -1 \\ |\bolda^2| = 1, \ldots,\p - 1-|\boldb| \\ \bolda \in \sh(\bolda^1,\bolda^2)}}\!\!\!\! (\bfH^{k_1}_{\bolda^1;s} \bfH^{k_2}_{(\bolda^2,\boldb);s} \bfX^\bolda_{su}\bfX^\boldb_{ut}) \\
			&\mathrel{\phantom{=}} + \!\!\!\! \sum_{\substack{(\boldb^1,\boldb^2) \in A^p_1 \\ |\bolda^1| = 0,\ldots,\p - 1 - |\boldb^1| \\ |\bolda^2| = 0,\ldots,\p - 1 - |\boldb^2| \\ \bolda \in \sh(\bolda^1,\bolda^2) \\\boldb \in \osh (\boldb^1,\boldb^2) }} \!\!\!\! (\bfH^{k_1}_{(\bolda^1,\boldb^1);s} \bfH^{k_2}_{(\bolda^2,\boldb^2);s} \bfX^\bolda_{su} \bfX^\boldb_{ut}) \\
			&= \!\!\!\! \sum_{(\boldc^1,\boldc^2) \in A^p_1 } \!\!\!\! \bfH^{k_1}_{\boldc^1;s} \bfH^{k_2}_{\boldc^2;s} \!\!\!\! \sum_{\substack{l = 0,1,2 \\ \bolda^h = \boldc^h, \ h \leq l \\ (\bolda^h, \boldb^h) = \boldc^h, \ |\boldb^h| \geq 1, \ h \geq l+1 \\ \bolda \in \sh(\osh(\bolda^1,\ldots,\bolda^l),\sh(\bolda^{l+1},\ldots,\bolda^2))\\ \boldb \in \osh(\boldb^{l+1},\ldots,\boldb^2)}}\!\!\!\!\bfX^\bolda_{su}\bfX^{\boldb}_{ut} \\
			&= \mathrel{\phantom{=}} \!\!\!\! \sum_{(\boldc^1,\boldc^2) \in A^p_1} \!\!\!\!  \bfH^{k_1}_{\boldc^1;s} \bfH^{k_2}_{\boldc^2;s} \sum_{\substack{\boldc \in \osh(\boldc^1,\boldc^2) \\ (\bolda, \boldb) = \boldc}}\bfX^\bolda_{su}\bfX^{\boldb}_{ut} \\
			&= (\xhapprox)_{st}^{(k_1,k_2)}.
		\end{align*}
		where we have used \autoref{lem:orderedCompat} in the second-last identity and multiplicativity of $\bfX$ in the last. The case of $m \coloneqq |\boldk| \geq 3$ (which presupposes $\p \geq 3$) is handled similarly, but has to be distinguished from the previous case since the middle terms are not the same.
		\begin{align*}
			&\mathrel{\phantom{=}} \sum_{(\boldi, \boldj) = \boldk} (\xhapprox)^\boldi_{su} (\xhapprox)^\boldj_{ut} \\
			&= \!\!\!\!\sum_{\substack{(\bolda^1,\ldots,\bolda^m) \in A^p_1 \\ \bolda \in \osh(\bolda^1,\ldots,\bolda^m)}}\!\!\!\! \bfH^{k_1}_{\bolda^1;s} \cdots \bfH^{k_m}_{\bolda^m;s} \bfX^\bolda_{su} + \Big(\!\!\!\!\sum_{\substack{(\bolda^1,\ldots,\bolda^{m-1}) \in A^p_1 \\ \bolda \in \osh(\bolda^1,\ldots,\bolda^{m-1})}}\!\!\!\!\bfH^{k_1}_{\bolda^1;s} \cdots \bfH^{k_{m-1}}_{\bolda^{m-1};s} \bfX^\bolda_{su} \Big) H^{k_m}_{ut}  \\
			&\mathrel{\phantom{=}} + \sum_{l = 2}^{m-2} \Big(\!\!\!\!\sum_{\substack{(\bolda^1,\ldots,\bolda^l) \in A^p_1 \\ \bolda \in \osh(\bolda^1,\ldots,\bolda^l)}}\!\!\!\! \bfH^{k_1}_{\bolda^1;s} \cdots \bfH^{k_l}_{\bolda^l;s} \bfX^\bolda_{su} \Big) \Big(\!\!\!\!\sum_{\substack{(\boldb^{l+1},\ldots,\boldb^m) \in A^p_1 \\ \boldb \in \osh(\boldb^{l+1},\ldots,\boldb^m)}}\!\!\!\! \bfH^{k_{l+1}}_{\boldb^{l+1};u} \cdots \bfH^{k_m}_{\boldb^m;u} \bfX^\boldb_{ut} \Big) \\
			&\mathrel{\phantom{=}}+  H^{k_1}_{su} \Big(\!\!\!\!\sum_{\substack{(\boldb^2,\ldots,\boldb^m) \in A^p_1 \\ \boldb \in \osh(\boldb^2,\ldots,\boldb^m)}}\!\!\!\!\bfH^{k_2}_{\boldb^2;u} \cdots \bfH^{k_m}_{\boldb^m;u} \bfX^\boldb_{ut} \Big) + \!\!\!\!\sum_{\substack{(\boldb^1,\ldots,\boldb^m) \in A^p_1 \\ \boldb \in \osh(\boldb^1,\ldots,\boldb^m)}}\!\!\!\! \bfH^{k_1}_{\boldb^1;u} \cdots \bfH^{k_m}_{\boldb^m;u} \bfX^\boldb_{ut} \\
			&\approx \!\!\!\!\sum_{\substack{(\bolda^1,\ldots,\bolda^m) \in A^p_1 \\ \bolda \in \osh(\bolda^1,\ldots,\bolda^m)}}\!\!\!\! \bfH^{k_1}_{\bolda^1;s} \cdots \bfH^{k_m}_{\bolda^m;s} \bfX^\bolda_{su} \\
			&\mathrel{\phantom{=}}+ \Big(\!\!\!\!\sum_{\substack{(\bolda^1,\ldots,\bolda^{m-1}) \in A^p_1 \\ \bolda \in \osh(\bolda^1,\ldots,\bolda^{m-1})}}\!\!\!\!\bfH^{k_1}_{\bolda^1;s} \cdots \bfH^{k_{m-1}}_{\bolda^{m-1};s} \bfX^\bolda_{su} \Big) \Big(\!\!\!\!\sum_{\substack{|\boldb^m| = 1,\ldots, \p -1 \\ |\bolda^m| = 0, \ldots,\p - 1-|\boldb|}}\!\!\!\!\bfH^{k_m}_{(\bolda^m,\boldb^m);s} \bfX_{su}^{\bolda^m} \bfX^{\boldb^m}_{ut} \Big)  \\
			&\mathrel{\phantom{=}} + \sum_{l = 2}^{m-2} \Big[ \Big(\!\!\!\!\sum_{\substack{(\bolda^1,\ldots,\bolda^l) \in A^p_1 \\ \bolda \in \osh(\bolda^1,\ldots,\bolda^l)}}\!\!\!\! \bfH^{k_1}_{\bolda^1;s} \cdots \bfH^{k_l}_{\bolda^l;s} \bfX^\bolda_{su} \Big) \\
			&\mathrel{\phantom{=}} \cdot \Big(\!\!\!\!\sum_{\substack{(\boldb^{l+1},\ldots,\boldb^m) \in A^p_1 \\ |\bolda^h| = 0,\ldots,\p - 1 - |\boldb^h|, \ h \geq l+1  \\ \bolda \in \sh(\bolda^{l+1},\ldots,\bolda^m) \\ \boldb \in \osh(\boldb^{l+1},\ldots,\boldb^m)}}\!\!\!\! \bfH^{k_{l+1}}_{(\bolda^{l+1},\boldb^{l+1});s} \cdots \bfH^{k_m}_{(\bolda^m,\boldb^m);s} \bfX^\bolda_{su} \bfX^\boldb_{ut} \Big)\Big] \\
			&\mathrel{\phantom{=}}+  \Big( \sum_{|\bolda^1| = 1}^{\p - 1} \bfH_{\bolda^1;s}^{k_1} \bfX^{\bolda^1}_{su} \Big) \Big(\!\!\!\!\sum_{\substack{(\boldb^2,\ldots,\boldb^m) \in A^p_1 \\ |\bolda^h| = 0,\ldots,\p - 1 - |\boldb^h|, \ h \geq 2  \\ \bolda \in \sh(\bolda^{2},\ldots,\bolda^m) \\ \boldb \in \osh(\boldb^{2},\ldots,\boldb^m)}}\!\!\!\! \bfH^{k_{2}}_{(\bolda^{2},\boldb^{2});s} \cdots \bfH^{k_m}_{(\bolda^m,\boldb^m);s} \bfX^\bolda_{su} \bfX^\boldb_{ut} \Big) + \\
			&+\sum_{\substack{(\boldb^1,\ldots,\boldb^m) \in A^p_1 \\ |\bolda^h| = 0,\ldots,\p - 1 - |\boldb^h|  \\ \bolda \in \sh(\bolda^1,\ldots,\bolda^m) \\ \boldb \in \osh(\boldb^1,\ldots,\boldb^m)}}\!\!\!\! \bfH^{k_1}_{(\bolda^1,\boldb^1);s} \cdots \bfH^{k_m}_{(\bolda^m,\boldb^m);s} \bfX^\bolda_{su} \bfX^\boldb_{ut} \\
			&= \!\!\!\! \sum_{(\boldc^1,\ldots,\boldc^m) \in A^p_1} \!\!\!\! \bfH^{k_1}_{\boldc^1;s} \cdots \bfH^{k_m}_{\boldc^m;s} \!\!\!\! \sum_{\substack{l = 0,\ldots,m \\ \bolda^h = \boldc^h, \ h \leq l \\ (\bolda^h, \boldb^h) = \boldc^h, \ |\boldb^h| \geq 1, \ h \geq l+1 \\ \bolda \in \sh(\osh(\bolda^1,\ldots,\bolda^l),\sh(\bolda^{l+1},\ldots,\bolda^m))\\ \boldb \in \osh(\boldb^{l+1},\ldots,\boldb^m)}}\!\!\!\!\bfX^\bolda_{su}\bfX^{\boldb}_{ut} \\
			&= \!\!\!\! \sum_{(\boldc^1,\ldots,\boldc^m) \in A^p_1} \!\!\!\! \bfH^{k_1}_{\boldc^1;s} \cdots \bfH^{k_m}_{\boldc^m;s} \sum_{\substack{\boldc \in \osh(\boldc^1,\ldots,\boldc^m) \\ (\bolda, \boldb) = \boldc}}\bfX^\bolda_{su}\bfX^{\boldb}_{ut} \\
			&= (\xhapprox)_{st}^{\boldk}.
		\end{align*}
		
		We proceed with the proof of geometricity. Again using \eqref{eq:generalApprox} we have
		\begin{align*}
			&\mathrel{\phantom{\approx}}(\xhapprox)^\boldi_{st} (\xhapprox)^\boldj_{st} \\
			&\approx \!\!\!\!\sum_{\substack{|\bolda^1|, \ldots , |\bolda^m|,|\boldb^1|, \ldots , |\boldb^n| \geq 1 \\ |\bolda^1|+ \ldots + |\bolda^m| + |\boldb^1|+ \ldots + |\boldb^n| \leq \p \\ \bolda \in \osh(\bolda^1,\ldots,\bolda^m) \\ \boldb \in \osh(\boldb^1,\ldots,\boldb^n) }}\!\!\!\!\bfH^{i_1}_{\bolda^1;s} \cdots \bfH^{i_m}_{\bolda^m;s} \bfH^{j_1}_{\boldb^1;s} \cdots \bfH^{j_n}_{\boldb^n;s} \bfX^\bolda_{st} \bfX^\boldb_{st} \\
			&= \!\!\!\!\sum_{\substack{|\bolda^1|, \ldots , |\bolda^m|,|\boldb^1|, \ldots , |\boldb^n| \geq 1 \\ |\bolda^1|+ \ldots + |\bolda^m| + |\boldb^1|+ \ldots + |\boldb^n| \leq \p }}\!\!\!\!\bfH^{i_1}_{\bolda^1;s} \cdots \bfH^{i_m}_{\bolda^m;s} \bfH^{j_1}_{\boldb^1;s} \cdots \bfH^{j_n}_{\boldb^n;s} \\
			&\mathrel{\phantom{=}} \cdot \sum_{\boldc \in \sh(\osh(\bolda^1,\ldots,\bolda^m), \osh(\boldb^1,\ldots,\boldb^n))}\!\!\!\! \bfX^\boldc_{st} \\
			&= \sum_{\boldk \in \sh(\boldi, \boldj)} \sum_{(\boldc^1,\ldots,\boldc^{m+n}) \in A^p_1}\!\!\!\! \bfH^{k_1}_{\boldc^1;s} \cdots \bfH^{k_{m+n}}_{\boldc^{m+n};s} \!\!\!\!\sum_{\boldc \in \osh(\boldc^1,\ldots,\boldc^{m+n})}\!\!\!\! \bfX^\boldc_{st} \\
			&= \sum_{\boldk \in \sh(\boldi,\boldj)} (\xhapprox)_{st}^\boldk
		\end{align*}
		where we have used \autoref{lem:shsh} (with $m = 2$) in the second last identity. We may therefore apply \autoref{prop:geomAlm} to conclude the proof.
	\end{proof}
\end{thm}

This construction immediately yields a couple of important examples of rough path:
\begin{expl}[Pushforward of rough paths]
	Let $\bfX, F$ be as in \autoref{ex:FX}. We denote
	\begin{equation}
		F_*\bfX \coloneqq \uparrow_{\!\bfX} \!\! \overline{F}(X)
	\end{equation}
	and call it the \emph{pushforward} of $\bfX$ through $F$. This is a rough path with trace $F(X)$.
\end{expl}
\begin{expl}[Rough integrals as rough paths]\label{expl:roughIntLift}
	Let $\bfX, \bfH$ be as in \autoref{def:rint}. We denote $\boldsymbol \int \bfH \dif \bfX \coloneqq \uparrow_{\!\bfX} \!\! \overline \int \bfH \dif \bfX$.
\end{expl}

We can use these notions to reinterpret the following well-known fact about weakly geometric rough paths. Notice how, in particular, this implies that the rough integral of an exact 1-form is entirely determined by its trace. (Incidentally, arbitrary 1-forms do not require the whole rough path for the integral to be defined either, just the terms $(\odot_{n-1} \otimes \mathbbm 1)\bfX^n$ for $n = 1,\ldots,\p$.)
\begin{prop}[Change of variable formula]\label{prop:changeVarSimple}
	Let $\bfX$ be as above, $F \in C^\infty(V,W)$, then the following identity
	\begin{equation}
		\overline F(X) = F(X_0) + \overline \int \overline{DF} (X) \emph{d} \bfX
	\end{equation}
	holds in $\mathscr D_{\bfX}(W)$. Therefore, the corresponding identity in $\mathscr C^p_\omega([0,T],W)$ holds as well:
	\begin{equation}
		F_*\bfX = F(X_0) + \boldsymbol \int \overline{DF}(X) \emph{d} \bfX
	\end{equation}
	where in both cases the constant $F(X_0)$ is only added to the trace of the integral.
	\begin{proof}
		For the trace we have, by Taylor's formula and \autoref{prop:sym}
		\begin{align*}
			F(X)_{st} &\approx \sum_{n = 1}^\p \frac{1}{n!} \langle  D^nF(X_s), X^{\otimes n}_{st} \rangle \\ 
			&= \sum_{n = 1}^\p \langle  D^nF(X_s), \bfX^n_{st} \rangle \\
			&\approx \int_s^t \overline{DF}(X) \dif \bfX
		\end{align*}
		and therefore \cite[Theorem 3.3.1]{Lyo98} implies $F(X)_{0t} = \int_0^t \overline{DF}(X) \dif \bfX$. The other claims follow trivially.
	\end{proof}
\end{prop}
\begin{rem}
	A similar formula would hold for more general controlled paths, i.e.\
	\begin{equation}
		\bfH = H_0 + \int \bfH' \dif \bfX
	\end{equation}
	(together with its rough path counterpart, given by passing to the $\uparrow_{\!\bfX}$ on both sides), provided that we have a $\p$-th Gubinelli derivative, needed to define the controlled integrand $\bfH'$.
\end{rem}

We would now like to show that a path controlled by the lift of a controlled path is controlled by the original rough path in a canonical fashion. 
\begin{defn}[Change of controlling rough path]\label{def:change}
	Let $\bfX \in \mathscr C^p_\omega([0,T],V)$, $\bfH \in \mathscr D_{\bfX}(U)$, $S$ another vector space, $\boldsymbol H \coloneqq \xhtrue$ and $\bfK \in \mathscr D_{\boldsymbol H}(S)$. We then define
	\begin{equation}\label{eq:k*h}
		(\bfK * \bfH)_n \coloneqq \sum_{m = 1}^{n} \bfK_m \circ \bfH{}^{\boxtimes m} \circ \wDOshuffle^m|_{V^{\otimes n}}
	\end{equation}
\end{defn}
which for $n = 0$ reduces to $(\bfK * \bfH)_0 \coloneqq K$. In coordinates this means
\begin{equation}
	(\bfK * \bfH)^c_\boldc \coloneqq \!\!\!\! \sum_{\substack{m = 1,\ldots,|\boldc| \\  (\boldc^1,\ldots,\boldc^m) \in \osh^{-1}(\boldc) \\ |\boldc^1|,\ldots,|\boldc^m| \geq 1  }} \!\!\!\!  \bfK^c_\boldk \bfH^{k_1}_{\boldc^1} \cdots \bfH^{k_m}_{\boldc^m}, \quad |\boldc| \geq 1
\end{equation}
and $(\bfK * \bfH)^c_{()} = K^c$.

In \autoref{prop:K*H} below we show that this defines a controlled path. The next example features a case in which the reduced ordered shuffle coproduct can be replaced with its unordered counterpart; this however is not the general case.
\begin{expl}\label{expl:FH}
	Let $\bfX,\bfH,\boldsymbol H, \bfK$ be as above, with $\bfK \coloneqq \overline F(H)$ for $F \in C^\infty(U,S)$. Then, since $\bfK_m$ is symmetric we may rewrite \eqref{eq:k*h} by using the unordered shuffle coproduct: by \autoref{thm:osym}
	\begin{equation}
		(\overline F(H) * \bfH)_n = \sum_{m \geq 1} \frac{1}{m!} D^m F(H) \circ \bfH^{\boxtimes m} \circ \wDshuffle^m|_{V^{\otimes n}}, \quad n \geq 1
	\end{equation}
	or in coordinates
	\begin{equation}
		(\overline F(H) * \bfH)^c_\boldc \coloneqq \!\!\!\! \sum_{\substack{m = 1,\ldots,|\boldc|  \\ (\boldc^1,\ldots,\boldc^m) \in \sh^{-1}(\boldc) \\ |\boldc^1|,\ldots,|\boldc^m| \geq 1 }} \!\!\!\! \frac{1}{m!}  \partial_\boldk F^c(H) \bfH^{k_1}_{\boldc^1} \cdots \bfH^{k_m}_{\boldc^m}, \quad |\boldc| \geq 1
	\end{equation}
	When $\bfH$ is also given by a smooth function this is known as the Faà di Bruno formula for the higher derivatives of a composition of functions
	\begin{equation}\label{eq:faa}
		\partial_{\boldc}(F \circ G)^c(X) \coloneqq \!\!\!\! \sum_{\substack{m = 1,\ldots,|\boldc|  \\ (\boldc^1,\ldots,\boldc^m) \in \sh^{-1}(\boldc) \\ |\boldc^1|,\ldots,|\boldc^m| \geq 1 }} \!\!\!\! \frac{1}{m!}  \partial_\boldk F^c(H) \partial_{\boldc^1}G^{k_1} \cdots \partial_{\boldc^m}G^{k_m}(X).
	\end{equation}
	We will denote the $\bfX$-controlled path $\overline F(H) * \bfH \coloneqq F_*\bfH \in \mathscr D_{\bfX}(S)$ and call it the \emph{pushforward} of $\bfH$ through $F$. Note how this is distinct from $\overline F(H) \in \mathscr D_{\boldsymbol H}(S)$, and note how $F_* \overline X = \overline F(X)$ where $\overline X$ denotes the controlled path $X$ with zero Gubinelli derivatives. An easy application of \autoref{prop:changeVarSimple} shows the following change of variable formula for controlled paths:
	\begin{equation}
		F_*\overline H = \bigg( F(H_0) + \overline \int \overline{DF}(H) \dif (\uparrow_{\!\bfX} \!\! \bfH ) \bigg) * \bfH.
	\end{equation}
\end{expl}

\begin{prop}\label{prop:K*H}
	The map $\bfK * \bfH \in \mathcal L(T^{\p-1}(V),S)$ of \autoref{def:change} is an element of $\mathscr D_{\bfX}(S)$.
	\begin{proof}
		Clearly $\bfK * \bfH \in \mathcal C^p_\omega([0,T],\mathcal L(T^{\p - 1}(V),U))$, so it remains to show \eqref{eq:contrDefCoords}. We preliminarily write
		\begin{equation}\label{eq:Kcj}
			\begin{split}
				\bfK^c_{\boldj;st} &\approx_{\p - |\boldj|} \sum_{|\boldi| = 1}^{\p - |\boldj| - 1} \bfK^c_{(\boldi, \boldj);s} \boldsymbol H^{\boldi}_{st} \\
				&\approx \!\!\!\!\sum_{\substack{m = 1,\ldots, \p - |\boldj| - 1 \\   (\bolda^1,\ldots,\bolda^m) \in A^p_1 \\ \bolda \in \osh(\bolda^1,\ldots,\bolda^m)}}\!\!\!\! \bfK^c_{(\boldi, \boldj);s} \bfH^{i_1}_{\bolda^1;s} \cdots \bfH^{i_m}_{\bolda^m;s} \bfX^\bolda_{st}
			\end{split}
		\end{equation}
		and
		\begin{equation}\label{eq:Hjb}
			\bfH^j_{\boldd;st} \approx_{\p - |\boldd|} \sum_{|\bolda| = 1}^{\p - |\boldd| - 1} \bfH^j_{(\bolda,\boldd);s} \bfX^\bolda_{st}.
		\end{equation}
		Now, let $1 \leq |\boldb| \leq \p - 2$:
		\begin{align*}
			&\mathrel{\phantom{=}}(\bfK * \bfH)^c_{\boldb;t} \\
			&= \!\!\!\! \sum_{\substack{n = 1,\ldots,|\boldb| \\ (\boldb^1,\ldots,\boldb^n) \in A_1 \cap \osh^{-1}(\boldb) }} \!\!\!\!  \bfK^c_{\boldj;t} \bfH^{j_1}_{\boldb^1;t} \cdots \bfH^{j_n}_{\boldb^n;t} \\
			&= \!\!\!\! \sum_{\substack{n = 1,\ldots,|\boldb| \\ (\boldb^1,\ldots,\boldb^n) \in A_1 \cap \osh^{-1}(\boldb) }} \!\!\!\!  (\bfK^c_{\boldj;s}+\bfK^c_{\boldj;st}) (\bfH^{j_1}_{\boldb^1;s} + \bfH^{j_1}_{\boldb^1;st}) \cdots (\bfH^{j_n}_{\boldb^n;s} + \bfH^{j_n}_{\boldb^n;st}) \\
			&\approx_{\p - |\boldb|} \!\!\!\! \sum_{\substack{n = 1,\ldots,|\boldb| \\ (\boldb^1,\ldots,\boldb^n) \in A_1 \cap \osh^{-1}(\boldb)}} \!\!\!\!  \Big( \bfK^c_{\boldj;s} \bfH^{j_1}_{\boldb^1;s} \cdots \bfH^{j_n}_{\boldb^n;s} + \!\!\!\!\sum_{\substack{\epsilon_0,\ldots,\epsilon_n = 0,1 \\ (\epsilon_0,\ldots,\epsilon_n) \neq (0,\ldots,0)}} \!\!\!\! \xi_{\epsilon_0} \eta^1_{\epsilon_1} \cdots \eta^n_{\epsilon_n} \Big)
		\end{align*}
		where $\xi_0 \coloneqq \bfK^c_{\boldj;s}$, $\eta^l_0 \coloneqq \bfH^{j_l}_{\boldb^l;s}$, and $\xi_1,\eta_1^l$ are given by the right hand sides of \eqref{eq:Kcj} and \eqref{eq:Hjb} (with $\boldd \coloneqq \boldb^l$) respectively; the $\approx_{\p - |\boldb|}$ in the last line above holds since $|\boldj|,|\boldb^l| \leq |\boldb|$ and the $\epsilon_l$'s are not all zero. We can rewrite this as
		\begin{align*}
			&\mathrel{\phantom{\approx}}(\bfK * \bfH)^c_{\boldb;t} \\
			&\approx_{\p - |\boldb|} \!\!\!\!\sum_{\substack{n = 1,\ldots,|\boldb| \\ (\boldb^1,\ldots,\boldb^n) \in A_1 \cap \osh^{-1}(\boldb) \\ m = 0,\ldots, \p - n - 1 \\ (\bolda^1,\ldots,\bolda^m) \in A^p_1 \\ \boldd \in \osh(\bolda^1,\ldots,\bolda^m) \\ \bolda^{m+l} = 0,\ldots,\p - |\boldb^l| -1}}\!\!\!\! (\bfK^c_{(\boldi, \boldj);s} \bfH^{i_1}_{\bolda^1;s} \cdots \bfH^{i_m}_{\bolda^m;s} \bfX^\boldd_{st}) \cdot \\
			&\mathrel{\hphantom{\approx_{\p - |\boldb|} \!\!\!\!\sum_{\bolda^{m+l} = 0,\ldots,\p - |\boldb^l| -1}\!\!\!\!}} \cdot(\bfH^{j_1}_{(\bolda^{m+1},\boldb^1);s} \bfX^{\bolda^{m+1}}_{st}) \cdots (\bfH^{j_n}_{(\bolda^{m+n},\boldb^n);s} \bfX^{\bolda^{m+n}}_{st}) \\[1em]
			&\approx_{\p - |\boldb|} \!\!\!\!\!\!\!\!\sum_{\substack{q = 1,\ldots,\p-1 \\ n = 1,\ldots,|\boldb| \\ |\bolda^1|, \ldots, |\bolda^{q-n}| \geq 1 \\ |\bolda^1|+ \ldots + |\bolda^q| \leq \p - |\boldb| - 1 \\ \bolda \in  \sh(\osh(\bolda^1,\ldots,\bolda^{q-n}), \sh(\bolda^{q-n+1},\ldots,\bolda^q)) \\ (\boldb^1,\ldots,\boldb^n) \in A_1 \cap \osh^{-1}(\boldb) \\  \boldc^l = \bolda^l, \ l \leq q-n; \ \boldc^l = (\bolda^l,\boldb^l), \ l \geq q-n + 1}}\!\!\!\!\!\!\!\! \bfK^c_{\boldk;s} \bfH^{k_1}_{\boldc^1;s} \cdots \bfH^{k_{q}}_{\boldc^{q};s} \bfX_{st}^\bolda \\
			&= \sum_{|\bolda| = 0}^{\p - |\boldb| - 1} \bfK^c_{\boldk;s} \Big( \!\!\!\!\!\!\!\!\!\!\!\!\sum_{\substack{q = 1,\ldots,|\bolda| + |\boldb| \\ l = 0,\ldots,q \\ ((\bolda^1,\ldots,\bolda^l),(\bolda^{l+1},\ldots,\bolda^q)) \in (\osh^{-1},\sh^{-1})(\sh^{-1}(\bolda)) \\ (\boldb^{l+1},\ldots,\boldb^q) \in A_1 \cap \osh^{-1}(\boldb) \\ \boldc^h = \bolda^h, \ |\bolda^h| \geq 1, \ h \leq l;\ \boldc^h = (\bolda^h,\boldb^h), \ |\boldb^h| \geq 1, \ h \geq l + 1 }}\!\!\!\!\!\!\!\!\!\!\!\! \bfH^{k_1}_{\boldc^1;s} \cdots \bfH^{k_{q}}_{\boldc^{q};s} \Big) \bfX_{st}^\bolda \\
			&=\sum_{|\bolda| = 0}^{\p - |\boldb| - 1} \Big(\!\!\!\!\sum_{(\boldc^1,\ldots,\boldc^q) \in A_1 \cap \osh^{-1}(\bolda,\boldb) } \!\!\!\! \bfK^c_{\boldk;s} \bfH^{k_1}_{\boldc^1;s} \cdots \bfH^{k_{q}}_{\boldc^{q};s} \Big) \bfX_{st}^\bolda \\
			&= \sum_{|\bolda| = 0}^{\p - |\boldb| - 1} (\bfK * \bfH)^c_{(\bolda,\boldb);s} \bfX^\bolda_{st}
		\end{align*}
		as needed. The second-last identity above is given by \autoref{cor:dualReduced}. The proof of \eqref{eq:contrDefCoords} for $(\bfK * \bfH)_0 = K$ is a much simplified version of the proof above.
	\end{proof}
\end{prop}

Another application of the change of controlling path construction is a Leibniz rule for controlled paths.
\begin{defn}[Leibniz rule for controlled paths]
	Let $\bfX \in \mathscr C^p_\omega([0,T],V)$, $U_i$ be vector spaces for $i = 1,2,3$, $\bfH \in \mathscr D_{\bfX}(\mathcal L(U_1,U_2))$, $\bfK \in \mathscr D_{\bfX}(\mathcal L(U_2,U_3))$. Define $(\bfK \cdot \bfH)_0 \coloneqq K \circ H$ and for $n = 1,\ldots,\p - 1$
	\begin{equation}
		(\bfK \cdot \bfH)_n \coloneqq \times \circ (\bfK \boxtimes \bfH) \circ \Dshuffle
	\end{equation}
	where $\times \colon \mathcal L(U_2,U_3) \boxtimes \mathcal L(U_1,U_2) \to \mathcal L(U_1,U_3)$ is ordinary composition of linear maps.
\end{defn}
In coordinates
\begin{equation}\label{eq:leibnizCoords}
	(\bfK \cdot \bfH)_{\boldc}^{i \choose j} = \sum_{(\bolda,\boldb) \in \sh^{-1}(\boldc)} \bfK^{i \choose k}_{\bolda} \bfH^{k \choose j}_{\boldb} .
\end{equation}
Here we are using the notation $h \choose l$ to denote indices in spaces of linear maps, the upper index referring to the codomain and the lower to the domain; this allows us to use the Einstein convention on such superscripts. When the domain coincides with $V$ (i.e.\ it is a controlled integrand) we will often place it after the bottom tuple, e.g.\ $\bfH^{k \choose \delta}_\boldc = \bfH^k_{(\boldc, \delta)}$, as done previously. The presence of the unreduced $\Dshuffle$ in the previous definition might seem strange at first, but it is easily justified as follows:
\begin{prop}\label{prop:leibniz}
	$\bfK \cdot \bfH$ defines an element of $\mathscr D_{\bfX}(\mathcal L(U_1,U_3))$ which coincides with $ \overline\times(K,H) * (\bfK,\bfH)$. In particular, if $\bfH = \overline A(X)$, $\bfK = \overline B(X)$ for smooth functions $A,B$ then 
	\begin{equation}
		\bfH \cdot \bfK = \overline{A(\cdot)B(\cdot)}(X)
	\end{equation}
	with $A(\cdot)B(\cdot)$ denoting the function $x \mapsto A(x)B(x)$.
	\begin{proof}
		We apply \autoref{expl:FH} with the function $\times$ and the controlled path $(\bfK,\bfH) \in \mathscr D_{\bfX} (\mathcal L(U_2,U_3) \oplus \mathcal L(U_1,U_2))$. Denoting by $h \choose \cdot$ coordinates for the second direct summand and with $\cdot \choose k$ those for the first, we have
		\begin{align*}
			\times^{i \choose j} (\kappa,\eta) &= \eta^{i \choose l} \kappa^{l \choose j}\\
			\partial_{h \choose p} \times^{i \choose j} (\kappa,\eta) &= \delta^{ih} \kappa^{p \choose j}, \quad \partial_{q \choose k} \times^{i \choose j} (\kappa,\eta) = \delta^{jk} \eta^{i \choose q} \\
			\partial_{{h \choose p},{p \choose k}}\times^{i \choose j} (\kappa,\eta) &= \delta^{jk}\delta^{hi} = \partial_{{p \choose k},{h \choose p}}\times^{i \choose j} (\kappa,\eta)
		\end{align*}
		and all other derivatives vanish. Therefore
		\begin{align*}
			( \overline\times(K,H) * (\bfK,\bfH))^{i \choose j}_{\boldc} = \bfK^{i \choose l}_\boldc H^{l \choose j} + K^{i \choose l} \bfH^{l \choose j}_\boldc + \!\!\sum_{\substack{(\bolda, \boldb) \in \sh^{-1}(\boldc) \\ |\bolda|, |\boldb| \geq 1}}\!\! \bfK^{i \choose l}_\bolda \bfH^{l \choose k}_\boldb.
		\end{align*}
		The factor $1/2!$ is not present in the sum, since each non-vanishing second derivative is counted twice, as emphasised above. This expression coincides with \eqref{eq:leibnizCoords}. The last statement holds since $A(\cdot)B(\cdot) = \times (A,B)$.
	\end{proof}
\end{prop}

Next we define a notion of pullback for controlled integrands.
\begin{defn}
	Let $\bfX$ be as above, $F \in C^{\infty}(V,W)$, $\bfH \in \mathscr D_{F_*\bfX}(\mathcal L(W,U))$. Let
	\begin{equation}
		F^*\bfH \coloneqq (\bfH * \overline F(X)) \cdot \overline{DF}(X) \in \mathscr D_{\bfX}(\mathcal L(V,U)).
	\end{equation}
\end{defn}
We will not need the coordinate expression of the pullback of a controlled path, although it can still be derived as done in other cases. The next proposition reassures us of the compatibility and associativity of some of the operations defined up to now.
\begin{prop}\label{prop:properties}
	\begin{enumerate}
		\item Let $\bfX,\bfH,\boldsymbol H,\bfK$ be as in \autoref{expl:FH}, then
		\begin{equation}
			\uparrow_{\!\bfX}\!\!(\bfK * \bfH) = \uparrow_{\!\boldsymbol H}\!\! \bfK \in \mathscr C^p_\omega([0,T],S).
		\end{equation}
		In particular lifting commutes with pushforwards 
		\begin{equation}
			F_* (\uparrow_{\!\bfX} \!\!\bfH) = \uparrow_{\!\bfX} \!\! (F_* \bfH)
		\end{equation}
		and furthermore $
		F_*(G_*\bfX) = (F \circ G)_* \bfX$ for appropriately valued smooth functions $F,G$;
		\item $(\overline J *\bfK) *\bfH \eqqcolon \overline J *\bfK *\bfH \coloneqq \overline J *(\bfK *\bfH)$ for appropriately valued controlled paths $\bfH,\bfK,\overline J$, and in particular $F_*G_* \bfH = (F \circ G)_* \bfH$ for appropriately valued smooth functions $F,G$;
		\item $(\overline J \cdot \bfK) * \bfH = (\overline J * \bfH) \cdot (\bfK * \bfH)$ for appropriately valued controlled paths $\bfH,\bfK,\overline J$; in particular, taking $\overline J = \overline A(H)$, $\bfK = \overline B(H)$ we have $(A(\cdot)B(\cdot))_*\bfH = A_*\bfH \cdot B_*\bfH$;
		\item $(\overline J \cdot \overline K) \cdot \overline H \eqqcolon \overline J \cdot \overline K \cdot \overline H \coloneqq \overline J \cdot (\overline K \cdot \overline H)$ for appropriately valued controlled paths $\bfH,\bfK,\overline J$;
		\item $F^*(G^* \bfH) = (G \circ F)^* \bfH$ for appropriately valued smooth maps $F,G$.
	\end{enumerate}
	\begin{proof}
		We begin with 1.:
		\begin{align*}
			&\mathrel{\phantom{=}}\upharpoonleft_{\!\bfX}\!\!(\bfK * \bfH)^{(c_1,\ldots,c_m)}_{st} \\
			&= \sum_{\substack{(\boldc^1,\ldots,\boldc^m) \in A^p_1 \\ \boldc \in \osh(\boldc^1,\ldots,\boldc^m)}} (\bfK * \bfH)^{c_1}_{\boldc^1;s} \cdots (\bfK * \bfH)^{c_m}_{\boldc^m;s} \bfX_{st}^\boldc \\
			&= \!\!\!\!\sum_{\substack{(\boldc^1,\ldots,\boldc^m) \in A^p_1 \\ \boldc \in \osh(\boldc^1,\ldots,\boldc^m) \\ n_l = 1,\ldots, |\boldc^l| \\ (\boldc^{l1}, \ldots, \boldc^{ln_{\scaleto{l}{2.5pt}}}) \in \osh^{-1}(\boldc^l)}} \!\!\!\! (\bfK^{c_1}_{\boldk^1;s} \bfH^{k^1_1}_{\boldc^{11};s} \cdots \bfH^{k^1_{n_{\scaleto{1}{2.5pt}}}}_{\boldc^{1{n_{\scaleto{1}{2.5pt}}}};s}) \cdots (\bfK^{c_m}_{\boldk^m;s} \bfH^{k^m_1}_{\boldc^{m1};s} \cdots \bfH^{k^m_{n_{\scaleto{m}{1.8pt}}}}_{\boldc^{m{n_{\scaleto{m}{1.8pt}}}};s})  \bfX_{st}^\boldc  \\
			&= \sum_{\substack{(\boldk^1,\ldots,\boldk^m) \in A^p_1 \\ \boldh \in \osh(\boldk^1,\ldots,\boldk^m)}} \bfK^{c_1}_{\boldk^1;s}\cdots\bfK^{c_m}_{\boldk^m;s} \sum_{\substack{(\boldd^1,\ldots,\boldd^q) \in A^p_1 \\ \boldd \in \osh(\boldd^1,\ldots,\boldd^q)}} \bfH^{h_1}_{\boldd^1;s}\cdots \bfH^{h_q}_{\boldd^q;s} \bfX^\boldd_{st} \\
			&= (\upharpoonleft_{\!{\boldsymbol H}}\!\!\bfK)^{c_1,\ldots,c_m}
		\end{align*}
		and the statement follows by \cite[Theorem 3.3.1]{Lyo98}. As for the second statement
		\[
		F_*(\uparrow_{\! \bfX} \!\! \bfH) = \uparrow_{\!\boldsymbol H}\!\! \overline F (H) = \uparrow_{\!\bfX} \!\!(\overline F(H) * \bfH) =  \uparrow_{\!\bfX} \!\! (F_*\bfH)
		\]
		and
		\[
		F_*(G_*\bfX) = F_*(\uparrow_{\!\bfX}\!\! \overline G(X)) = \uparrow_{\!\bfX}\!\!(F_* \overline G(X)) = \uparrow_{\!\bfX}\!\!(F_* G_* \overline X) = \uparrow_{\!\bfX}\!\!((F \circ G)_* \overline X) 
		\]
		(where $\overline X$ is the $\bfX$-controlled path with trace $X$ and zero Gubinelli derivatives) and the conclusion is implied by 2.\ below.
		
		The proof of 2.\ is straightforward (with the second claim deduced from the Faà di Bruno formula \eqref{eq:faa}).
		
		As for 3., we have
		\begin{align*}
			((\overline J \cdot \bfK) * \bfH))_\boldc^{a \choose b} &= \!\!\!\! \sum_{\substack{m = 1,\ldots,|\boldc| \\  (\boldc^1,\ldots,\boldc^m) \in A_1 \cap \osh^{-1}(\boldc)   }} \!\!\!\!  (\overline J \cdot \bfK)^{a \choose b}_\boldk \bfH^{k_1}_{\boldc^1} \cdots \bfH^{k_m}_{\boldc^m} \\
			&= \!\!\!\! \sum_{\substack{m = 1,\ldots,|\boldc| \\  (\boldc^1,\ldots,\boldc^m) \in A_1 \cap \osh^{-1}(\boldc) \\ (\boldi, \boldj) \in \sh^{-1}(\boldk)}} \!\!\!\!  \overline J^{a \choose c}_{\boldi} \bfK^{c \choose b}_\boldj \bfH^{k_1}_{\boldc^1} \cdots \bfH^{k_m}_{\boldc^m}\\
			&= \!\!\!\! \sum_{\substack{l+q = 1,\ldots,|\boldc| \\  (\bolda, \boldb) \in \sh^{-1}(\boldc)\\ (\bolda^1,\ldots,\bolda^l) \in A_1 \cap \osh^{-1}(\bolda) \\  (\boldb^1,\ldots,\boldb^q) \in A_1 \cap \osh^{-1}(\boldb) }} \!\!\!\!  \overline J^{a \choose c}_{\boldi} \bfK^{c \choose b}_\boldj \bfH^{i_1}_{\bolda^1} \cdots \bfH^{i_l}_{\bolda^l} \bfH^{j_1}_{\boldb^1} \cdots \bfH^{j_q}_{\boldb^q}\\
			&=(\overline J * \bfH) \cdot (\bfK * \bfH)
		\end{align*}
		where the second last identity follows from \autoref{lem:shsh} (with $m = 2$).
		
		As for 4., it is easy to show, using associativity of composition and of $\Dshuffle$, that both sides coincide with $\times^3 \circ ( \overline J \boxtimes \bfK \boxtimes \bfH) \circ \Dshuffle^3$ where $\times^3$ denotes composition of three linear maps.
		
		Finally, 5.\ is shown as follows:
		\begin{align*}
			F^*G^*\bfH &= (G^* \bfH * \overline F(X)) \cdot \overline{DF}(X) \\
			&= (((\bfH * \overline G(F(X))) \cdot \overline{DG}(F(X))) * \overline F(X)) \cdot \overline{DF}(X) \\
			&= (\bfH * \overline G(F(X))* \overline F(X)) \cdot (\overline{DG}(F(X)) * \overline F(X)) \cdot \overline{DF}(X) \\
			&= ((\bfH * \overline{G \circ F}(X)) \cdot \overline{D(G \circ F)}(X) \\
			&= (G \circ F)^*\overline H.
		\end{align*}
		Here, we have used the previous points 2., 3.\ and 4.\ in the proposition, as well as the fact that
		\[
		\overline G(F(X)) * \overline F(X) = G_*\overline F(X) = G_* F_* \overline X = (G \circ F)_*\overline X = \overline{G \circ F}(X) 
		\]
		and similarly that 
		\begin{align*}
			(\overline{DG}(F(X)) * \overline F(X)) \cdot \overline{DF}(X) &= \overline{DG \circ F}(X) \cdot \overline{DF}(X)\\
			&= \overline{DG \circ F(\cdot) DF(\cdot)}(X)\\
			&= \overline{D(G \circ F)}(X)
		\end{align*}
		where we have used \autoref{prop:leibniz}.
	\end{proof}
\end{prop}

In the next theorem we prove the property, well-known in both ordinary and stochastic calculus, which allows to \say{substitute the differential}. This will be especially convenient when manipulating RDEs. We have only introduced the theory necessary to handle weakly geometric rough paths, and the theorem is therefore stated in this context, but one can expect this type of result to also hold true in other settings, such as It\^o calculus and branched rough paths.
\begin{thm}[Associativity of the rough integral]\label{thm:assoc}
	Let $\bfX \in \mathscr C^p_\omega([0,T],V)$, $\bfH \in \mathscr D_{\bfX}(\mathcal L(V,W))$, $I \coloneqq \int \bfH \emph{d} \bfX$, $\overline I$ and $\boldsymbol I$ respectively the canonical controlled and rough paths above $I$, $\bfK \in \mathscr D_{\boldsymbol I}(\mathcal L(W,U))$. Then
	\begin{equation}
		\bigg(\overline \int \bfK \emph{d} \boldsymbol I \bigg) * \overline I = \overline \int (\bfK * \overline I) \cdot \bfH \emph{d} \bfX
	\end{equation}
	and therefore 
	\begin{equation}
		\boldsymbol \int \bfK \emph{d} \boldsymbol I = \boldsymbol \int (\bfK * \overline I) \cdot \bfH \emph{d} \bfX.
	\end{equation}
	\begin{proof}
		In this proof we will denote, for a tuple $\boldc$, $\gamma^\cdot$ its last entry and $\boldc^-$ the tuple obtained by removing $\gamma^\cdot$, so $\boldc = (\boldc^-,\gamma^\cdot)$. Moreover, we will interchangeably use the two indexing notations for controlled integrands, e.g.\ $\bfH^k_\boldc = \bfH^{k \choose \gamma^\cdot}_{\boldc^-}$. For $|\boldc| \geq 1$ we then have
		\begin{align*}
			\bigg( \overline \int (\bfK * \overline I) \cdot \bfH \dif \bfX \bigg)^c_\boldc &= ((\bfK * \overline I) \cdot \bfH )^c_\boldc \\ &= \sum_{(\bolda, \boldb) \in  \sh^{-1}(\boldc^-)} (\bfK * \overline I)^{c \choose h}_\bolda \bfH^{h \choose \gamma^\cdot}_\boldb \\
			&=\!\! \sum_{\substack{(\bolda, \boldb) \in \sh^{-1}(\boldc^-) \\ (\bolda^1,\ldots,\bolda^n) \in A_1 \cap \osh^{-1}(\bolda)}}\!\! \bfK^{c \choose h}_\boldk \bfH^{k_1}_{\bolda^1} \cdots \bfH^{k_n}_{\bolda^n} \bfH^{h \choose \gamma^\cdot}_\boldb \\
			&=\!\! \sum_{\substack{(\bolda, \boldb) \in \sh^{-1}(\boldc^-) \\ (\bolda^1,\ldots,\bolda^n) \in A_1 \cap \osh^{-1}(\bolda)}}\!\! \bfK^c_{(\boldk,h)} \bfH^{k_1}_{\bolda^1} \cdots \bfH^{k_n}_{\bolda^n} \bfH^h_{(\boldb,\gamma^\cdot)} \\ 
			&= \!\! \sum_{\substack{m = 1,\ldots,|\boldc| \\ (\boldc^1,\ldots,\boldc^m) \in A_1 \cap \osh^{-1}(\boldc) }} \!\! \bfK^c_{\boldk} \bfH^{k_1}_{\boldc^1} \cdots  \bfH^{k_m}_{\boldc^m} \\
			&= \!\! \sum_{\substack{m = 1,\ldots,|\boldc| \\ (\boldc^1,\ldots,\boldc^m) \in A_1 \cap \osh^{-1}(\boldc) }} \!\! \bigg(\overline \int \bfK \dif \boldsymbol I \bigg)^c_\boldk \overline I^{k_1}_{\boldc^1} \cdots \overline I^{k_m}_{\boldc^m} \\
			&= \bigg(\bigg(\overline \int \bfK \dif \boldsymbol I \bigg) * \overline I \bigg)_\boldc^c.
		\end{align*}
		At the trace level, we have, through a similar argument
		\begin{align*}
			\bigg( \overline \int (\bfK * \overline I) \cdot \bfH \dif \bfX \bigg)^c_{\boldc;st} &\approx \sum_{|\boldc| = 1}^\p ((\bfK * \overline I) \cdot \bfH )^c_{\boldc;s} \bfX^\boldc_{st} \\
			&= \!\! \sum_{\substack{|\boldc| = 1,\ldots \p \\ m = 1,\ldots,|\boldc| \\ (\boldc^1,\ldots,\boldc^m) \in A_1 \cap \osh^{-1}(\boldc) }} \!\! \bfK^c_{\boldk;s} \bfH^{k_1}_{\boldc^1;s} \cdots  \bfH^{k_m}_{\boldc^m;s} \bfX^\boldc_{st} \\
			&\approx \int \bfK^c \dif \boldsymbol I.
		\end{align*}
		As for the statement at the level of rough paths, we have
		\begin{align*}
			\boldsymbol\int \bfK \dif \boldsymbol I &= \Big\uparrow_{\!\boldsymbol I} \overline \int \bfK \dif \boldsymbol I \\
			&= \Big\uparrow_{\!\bfX} \bigg[ \bigg(\overline \int \bfK \dif \boldsymbol I \bigg) * \overline I \bigg] \\
			&= \Big\uparrow_{\!\bfX} \bigg[ \overline \int (\bfK * \overline I) \cdot \bfH \dif \bfX \bigg] \\
			&= \boldsymbol \int (\bfK * \overline I) \cdot \bfH \dif \bfX
		\end{align*}
		where we have used 1.\ in \autoref{prop:properties} and the previous statement.
	\end{proof}
\end{thm}

The next result, for which geometricity is essential, tells us that $F_*$ and $F^*$ behave as adjoint operators under the rough integral pairing. Its proof is an immediate consequence of  \autoref{prop:changeVarSimple} and \autoref{thm:assoc}.
\begin{thm}[Pushforward-pullback adjunction]\label{thm:pushPull}
	Let $\bfX, \bfH, F$ be as above, then
	\begin{equation}
		\bigg(\overline\int \bfH \emph{d} F_*\bfX\bigg) * \overline{F}(X) = \overline \int F^*\bfH \emph{d} \bfX
	\end{equation}
	and therefore
	\begin{equation}
		\boldsymbol{\int} \bfH \emph{d} F_*\bfX = \boldsymbol \int F^*\bfH \emph{d} \bfX.
	\end{equation}
\end{thm}

Next we move on to the topic of rough differential equations (RDEs). We will introduce two equivalent notions of solution to an RDE. Given a field of linear maps $F \in C^\infty(W,\mathcal L(V,W))$ and a smooth map $g \in C^\infty(W,U)$ we define, for $y \in W$
\begin{equation}
	Fg(y) \coloneqq Dg(y) \circ F(y) \in \mathcal L(V,U)
\end{equation}
and inductively
\begin{equation}
	F^ng(y) \coloneqq F(\eta \mapsto F^{n-1} g(\eta))|_{\eta = y} \in \mathcal L(V,\mathcal L(V^{\otimes n-1},U)) = \mathcal  L(V^{\otimes n},U).
\end{equation}
In coordinates we denote 
\[
F_\gamma g^c(y) \coloneqq Fg(y)^c_\gamma = \partial_k g(y) F^k_\gamma(y) \ \Rightarrow \ F^n g(y)_{(\gamma_1,\ldots,\gamma_n)}^c = F_{\gamma_1} \cdots F_{\gamma_n}g^c(y).
\]
We will also use the compact notation $F_\boldc g^c(y)$ for the latter.
\begin{rem}\label{rem:rightToLeft}
	Note that $F_{\gamma_1} \cdots F_{\gamma_n}g^c(y)$ can be read right to left as well as left to right, i.e.\ it is equal to $F_{\gamma_1} \cdots F_{\gamma_{n-1}}(F_{\gamma_n}g^c)(y)$ for $n \geq 2$. This can be seen by induction on $n$ (with the quantifier $\forall g$ inside the inductive hypothesis). For $n = 2$ the statement is tautological. For the inductive step we have
	\begin{align*}
		F_{\gamma_1} \cdots F_{\gamma_{n+1}}g^c(y) &= F_{\gamma_1}(F_{\gamma_2} \cdots F_{\gamma_{n+1}}g^c)(y) \\
		&= F_{\gamma_1}(F_{\gamma_2} \cdots F_{\gamma_n}(F_{\gamma_{n+1}}g^c))(y) \\
		&=F_{\gamma_1} \cdots F_{\gamma_n}(F_{\gamma_{n+1}}g^c)(y)
	\end{align*}
	where in the second identity we have used the inductive hypothesis.
\end{rem}
\begin{defn}[Davie solution to an RDE]\label{def:davie}
	Let $F \in C^\infty(W,\mathcal L(V,W))$. A \emph{solution} to the RDE 
	\begin{equation}\label{eq:RDE}
		\dif Y = F(Y) \dif \bfX, \quad Y_0 = y_0
	\end{equation}
	is a path $Y \in C([0,T],W)$ starting at $y_0$ with the property that for all $g \in C^\infty (W)$.
	\begin{equation}\label{eq:davieSol}
		g(Y)_{st} \approx \sum_{n = 1}^\p \langle F^ng(Y_s), \bfX^n_{st} \rangle.
	\end{equation}
\end{defn}
\begin{prop}[Gubinelli solution to an RDE]\label{prop:gubDavie}
	Let $Y$ be a solution to \eqref{eq:RDE}. Then
	\begin{equation}\label{eq:gubY}
		\overline Y \coloneqq (Y,F(Y) = F\mathbbm 1 (Y), \ldots, F^{\p-1}\mathbbm 1 (Y)) \in \mathscr D_{\bfX}(W)
	\end{equation}
	and moreover
	\begin{equation}\label{eq:gubSol}
		\overline Y = y_0 + \overline\int F_*\overline Y \edif \bfX.
	\end{equation}
	Conversely, if an $\bfX$-controlled controlled path satisfying the above identity, its trace satisfies \autoref{def:davie}.
\end{prop}
In order to prove this proposition we will make use of the following
\begin{lem}\label{lem:FY}
	The $\bfX$-controlled path $\overline Y$ has the form \eqref{eq:gubY} if and only if $\overline Y_n = (F_*\overline Y)_{n-1}$ for $n \geq 1$. Moreover, in this case
	\[
	((Fg)_*\overline Y)_{n-1} = F^n g(Y),\quad n \geq 1
	\]
	with $g$ as above.
	\begin{proof}
		If we prove 
		\[
		(Fg)_*(Y,F(Y), \ldots, F^{\p-1}\mathbbm 1 (Y))_{n-1} = F^ng(Y)
		\]
		we will have shown both the second statement and the \say{only if} part of the first (just choose $g = \mathbbm 1$). In order to show this, we first establish the identity
		\begin{equation}\label{eq:FFF}
			F_\boldc g^c(y) = \sum_{\substack{(\boldc^1,\ldots,\boldc^m) \in \osh^{-1}(\boldc) \\ |\boldc^1|, \ldots, |\boldc^m| \geq 1}} \partial_\boldk g^c F_{\boldc^1} \mathbbm 1^{k_1} \cdots F_{\boldc^m} \mathbbm 1^{k_m}(y).
		\end{equation}
		Note that this is not a closed form formula for $F^n g$, since iterated compositions of the vector fields $F_\boldc$ also appear on the right hand side, but it will be useful for us nonetheless.  We proceed by induction on $|\boldc|$. For $|\boldc| = 1$ there is nothing to show. For the inductive step, using \autoref{rem:rightToLeft} we have
		\begin{align*}
			&\mathrel{\phantom{=}}F_{(\gamma_1,\ldots,\gamma_{n+1})} g^c (y) \\
			&= F_{\gamma_1}(F_{\gamma_2,\ldots,\gamma_{n+1}} g^c)(y) \\
			&= F_{\gamma_1} \!\!\!\!  \!\!\!\!  \!\!\!\! \sum_{\substack{(\boldc^1,\ldots,\boldc^m) \in A_1 \cap \osh^{-1}(\gamma_2,\ldots,\gamma_{n+1})}}  \!\!\!\!  \!\!\!\! \!\!\!\! \partial_\boldk g^c F_{\boldc^1} \mathbbm 1^{k_1} \cdots F_{\boldc^m} \mathbbm 1^{k_m}(y) \\
			&= \!\!\!\!  \!\!\!\!  \!\!\!\! \sum_{\substack{(\boldc^1,\ldots,\boldc^m) \in A_1 \cap \osh^{-1}(\gamma_2,\ldots,\gamma_{n+1}) }}  \!\!\!\!  \!\!\!\! \!\!\!\! \big( \partial_{h,\boldk} g^c F_{\boldc^1} \mathbbm 1^{k_1} \cdots F_{\boldc^m} \mathbbm 1^{k_m}  \\
			&\mathrel{\phantom{= \!\!\!\!  \!\!\!\!  \!\!\!\! \sum_{(\boldc^1,\ldots,\boldc^m) \in A_1 \cap \osh^{-1}(\gamma_2,\ldots,\gamma_{n+1})}  \!\!\!\!  \!\!\!\! \!\!\!\!}} +\sum_{l = 1}^m \partial_{\boldk} g^c F_{\boldc^1} \mathbbm 1^{k_1} \cdots \partial_h (F_{\boldc^l} \mathbbm 1^{k_l}) \cdots F_{\boldc^m} \mathbbm 1^{k_m}  \big) F^h_{\gamma_1}(y) \\
			&= \!\!\!\!  \!\!\!\!  \!\!\!\! \sum_{\substack{(\boldc^1,\ldots,\boldc^m) \in A_1 \cap \osh^{-1}(\gamma_2,\ldots,\gamma_{n+1})}}  \!\!\!\!  \!\!\!\! \!\!\!\! \big( \partial_{h,\boldk} g^c F^h_{\gamma_1} F_{\boldc^1} \mathbbm 1^{k_1} \cdots F_{\boldc^m} \mathbbm 1^{k_m}  \\
			&\mathrel{\phantom{= \!\!\!\!  \!\!\!\!  \!\!\!\! \sum_{\substack{(\boldc^1,\ldots,\boldc^m) \in A_1 \cap  \osh^{-1}(\gamma_2,\ldots,\gamma_{n+1})}}  \!\!\!\!  \!\!\!\! \!\!\!\!}} +\sum_{l = 1}^m \partial_{\boldk} g^c F_{\boldc^1} \mathbbm 1^{k_1} \cdots 
			F_{\gamma_1,\boldc^l} \mathbbm 1^{k_l} \cdots F_{\boldc^m} \mathbbm 1^{k_m}  \big)(y) \\
			&=  \!\!\!\!  \!\!\!\! \!\!\!\! \sum_{\substack{(\boldc^1,\ldots,\boldc^m) \in A_1 \cap \osh^{-1}(\gamma_1,\ldots,\gamma_{n+1})}}  \!\!\!\!  \!\!\!\! \!\!\!\! \partial_\boldk g^c F_{\boldc^1} \mathbbm 1^{k_1} \cdots F_{\boldc^m} \mathbbm 1^{k_m}(y).
		\end{align*}
		Now, for $n = |\boldc| \geq 1$ we have
		\begin{align*}
			((Fg)_*\overline Y)^c_{\boldc} &= \!\!\!\! \sum_{(\boldc^1,\ldots,\boldc^m) \in \osh^{-1}(\gamma_1,\ldots,\gamma_{n-1})}\!\!\!\! \partial_\boldk(F_{\gamma_n}g^c) F_{\boldc^1} \mathbbm 1^{k_1} \cdots F_{\boldc^m} \mathbbm 1^{k_m}(Y) \\
			&= F_{(\gamma_1,\ldots,\gamma_{n-1})} (F_{\gamma_n}g^c)(y) \\
			&= F_\boldc g^c(y)
		\end{align*}
		where we have used \autoref{rem:rightToLeft} and \eqref{eq:FFF}.
		
		We now show the \say{if} implication of the first statement. Namely, we need to show that if $\overline Y \in \mathscr D_{\bfX}(W)$ has the property that $\overline Y_n = (F_* \overline Y)_{n-1}$ for $n = 1,\ldots, \p$ then $\overline Y_n = F^n\mathbbm 1 (Y)$. We show this by induction on $n$. For $n = 1$ the assertion is obvious. For the inductive step we have
		\begin{align*}
			\overline Y^h_\boldc &=\!\!\!\!  \!\!\!\! \!\!\!\! \sum_{\substack{(\boldc^1,\ldots,\boldc^m) \in A_1 \cap \osh^{-1}(\boldc) }}\!\!\!\!  \!\!\!\! \!\!\!\! \partial_\boldk F^h(Y) 	\overline Y^{k_1}_{\boldc^1} \cdots \overline Y^{k_1}_{\boldc^1} \\
			&=\!\!\!\!  \!\!\!\! \!\!\!\! \sum_{\substack{(\boldc^1,\ldots,\boldc^m) \in A_1 \cap \osh^{-1}(\boldc)}}\!\!\!\!  \!\!\!\! \!\!\!\! \partial_\boldk F^h F_{\boldc^1} \mathbbm 1^{k_1} \cdots F_{\boldc^1}\mathbbm 1^{k_1}(Y) \\
			&= F_\boldc \mathbbm 1^h(Y)
		\end{align*} 
		where we have used \eqref{eq:FFF} and the inductive hypothesis.
	\end{proof}
\end{lem}	
\begin{proof}[Proof of \autoref{prop:gubDavie}]
	Let $Y$ be a Davie solution to the RDE. Taking $g$ in \eqref{eq:davieSol} to be $\mathbbm 1, F,\ldots, F^{\p - 2}\mathbbm 1$ proves that $\overline Y$ defined in \eqref{eq:gubY} is indeed an element of $\mathscr D_{\bfX}(W)$. By \autoref{lem:FY} we then have $F_*\overline Y = (F(Y),\ldots,F^{\p}\mathbbm 1 (Y))$ and by \autoref{def:rint}
	\[
	\int_s^t F_*\overline Y \dif \bfX \approx \langle (F_*\overline Y)_s, \bfX^{\geq 1}_{st} \rangle \approx Y_{st}
	\]
	again by the Davie definition. Since both the left and right hand sides are increments of paths, we conclude by \cite[Theorem 3.3.1]{Lyo98} that identity must hold. Therefore, since $Y_0 = y_0$, \eqref{eq:gubSol} holds at the trace level, and for $n \geq 1$
	\[
	\bigg( y_0 + \overline\int F_*\overline Y \dif \bfX \bigg)_n = (F_*\overline Y)_{n-1} = \overline Y_n.
	\]
	
	Conversely, assume that there exists some $\overline Y \in \mathscr D_{\bfX}(W)$ s.t.\ \eqref{eq:gubSol} holds: this implies that for $m \geq 0$
	\[
	(F_*\overline Y)_m = \bigg( y_0 + \overline\int F_*\overline Y \dif \bfX \bigg)_{m+1} = \overline Y_{m+1}
	\]
	and therefore by \autoref{lem:FY} $\overline Y$ must have the form \eqref{eq:gubY}. Finally, $Y_0 = y_0$ and for $g \in C^\infty(W)$ and $\bfY \coloneqq \uparrow_{\!\bfX}\!\!\overline Y$
	\begin{align*}
		g(Y)_{st} &= \int_s^t \overline{Dg}(Y) \dif \bfY \\ 
		&= \int_s^t (\overline{Dg}(Y) * \overline Y) \cdot (\overline F(Y) * \overline Y) \dif \bfX \\
		&= \int_s^t (\overline{Dg}(Y) \cdot \overline F(Y)) * \overline Y \dif \bfX \\
		&= \int_s^t \overline{(Dg(\cdot) F(\cdot) )}(Y) * \overline Y \dif \bfX \\
		&= \int_s^t (Fg)_*\overline Y  \dif \bfX \\
		&\approx \sum_{n = 1}^\p \langle F^ng(Y_s), \bfX^n_{st} \rangle
	\end{align*}
	where we have used \autoref{thm:assoc}, \autoref{prop:properties} and \autoref{lem:FY}. This concludes the proof.
\end{proof}
The above proposition tells us that once we have the solution in the sense of \autoref{def:davie} we can obtain an $\bfX$-controlled path, and thus by \autoref{def:lift} a rough path. If we want to emphasise the existence of these superstructures we will write
\begin{equation}
	\dif \overline Y = F(Y) \dif \bfX \quad \text{and} \quad \dif \boldsymbol Y = F(Y) \dif \bfX, \quad Y_0 = y_0
\end{equation}
i.e.\ $\bfY \coloneqq \uparrow_{\!\bfX}\!\!\overline Y$. Notice that the initial condition only involves the trace.

The next result will be instrumental in defining RDEs on manifolds in a coordinate-invariant manner.
\begin{thm}[Change of variable formula for RDE solutions]\label{thm:changeVarRDEs}
	Let $\bfX,F, \bfY$ be as above, $g \in C^\infty (W,U)$. Then $(\boldsymbol Y,g_*\boldsymbol Y)$ jointly solve the RDE
	\begin{equation}
		\edif\begin{pmatrix} \boldsymbol Y \\ \boldsymbol Z
		\end{pmatrix} = \begin{pmatrix} F(Y) \\ Dg(Y)F(Y)
		\end{pmatrix} \edif \bfX.
	\end{equation}
	In particular, if $g$ is invertible, Defining $C^\infty(U,\mathcal L(V,U)) \ni F_g(z) \coloneqq Dg(g^{-1}(Z))F(g^{-1}(Z))$, $g_*\bfY$ coincides with the rough path solution to
	\begin{equation}
		\edif \boldsymbol Z = F_g(Z) \edif \bfX.
	\end{equation}
	\begin{proof}
		Using \autoref{prop:changeVarSimple}, \autoref{thm:assoc} and \autoref{prop:properties} we have
		\begin{align*}
			\dif (g_* \bfY) &= y_0 + \boldsymbol \int \overline{Dg}(Y) \dif \bfY \\
			&= y_0 + \boldsymbol \int Dg_* \overline Y \cdot F_* \overline Y \dif \bfX \\
			&= y_0 + \boldsymbol \int (Dg(\cdot ) F(\cdot))_*\overline Y \dif \bfX.
		\end{align*}
		This proves the first claim; as for the second, we continue 
		\begin{align*}
			\dif (g_* \bfY) &=  \boldsymbol\int (Dg(\cdot ) F(\cdot))_*\overline Y \dif \bfX \\
			&=  \boldsymbol\int (Dg(g^{-1}(\cdot)) F(g^{-1}(\cdot)))_* g_*\overline Y \dif \bfX
		\end{align*}
		where we have again used \autoref{prop:properties}. This concludes the proof.
	\end{proof}
\end{thm}
The following theorem is proved in \cite[Corollary 2.17, Theorem 4.2]{CDL15} in the case of $2 \leq p < 3$ and the proof carries over to the general case:
\begin{thm}[Local existence and uniqueness]\label{thm:localE}
	Precisely one of the following two possibility holds w.r.t.\ \eqref{eq:RDE}
	\begin{enumerate}
		\item A solution on $[0,T]$ exists;
		\item There exists an $S \leq T$ and a solution on $[0,S)$, with $Y_{[0,S)}$ not contained in any compact set of $\bbR^e$.
	\end{enumerate}
	Moreover, in either case, the solution is unique on the interval on which it is defined.
\end{thm}
\begin{expl}[Non-autonomous RDEs]\label{expl:nonAuto}
	We can define RDEs that also depend on the driving signal, by \say{doubling the variables} i.e.\
	\begin{equation}
		\dif \bfY = F(Y,X) \dif \bfX \ \stackrel{\text{def}}{\Longleftrightarrow} \ \dif\begin{pmatrix} \bfX \\ \bfY
		\end{pmatrix} = \begin{pmatrix} \mathbbm 1 \\ F(Y,X)
		\end{pmatrix} \dif \bfX.
	\end{equation}
	This will be important when defining RDEs on manifolds, driven by a manifold-valued rough path.
\end{expl}

\section{Weakly geometric rough paths on manifolds}\label{sec:mfds}
In this section we show how our algebraic framework for weakly geometric rough paths can be deployed to transfer the theory to the manifold setting. In this section $M$ will be a smooth $m$-dimensional manifold and $TM$ its tangent bundle. The following definition is similar in spirit to the the one provided in \cite{BL15}.
\begin{defn}[Manifold-valued rough path]\label{def:rpM}
	Given a smooth atlas $\{\varphi \colon A_\varphi \to \bbR^m\}_\varphi$ of $M$, an $M$-valued $p$-\emph{weakly geometric rough path} controlled by $\omega$ on $[0,T]$, $\bfX \in \mathscr C^p_\omega([0,T], M)$, consists of a collection of rough paths ${^\varphi \!}\bfX|_{[a_\varphi,b_\varphi]} \in \mathscr C^p_\omega([a_\varphi,b_\varphi],\bbR^m)$ for all possible collections of intervals $[a_\varphi,b_\varphi]$ (indexed by the chart $\varphi$ in the atlas) s.t.\ $\text{Im}({^\varphi\!}X|_{[a_\varphi,b_\varphi]}) \subseteq \text{Range}(\varphi)$, and that for each pair of intervals $[a_\varphi,b_\varphi]$, $[a_\psi,b_\psi]$
	\begin{equation}\label{eq:compcond}
		(\psi \circ \varphi^{-1})_* {^\varphi \!}\bfX = {^\psi \!}\bfX \in \mathscr C^p_\omega([a_\varphi,b_\varphi] \cap [a_\psi,b_\psi],\bbR^m).
	\end{equation}
	The \emph{trace} of $\bfX$ is the path $t \mapsto X_t \coloneqq \varphi^{-1} ({^\varphi \!}X_t) \in M$ whenever $t \in [a_\varphi,b_\varphi]$ (independently of $\varphi$), $X \in \mathcal C^p([0,T],M)$.
\end{defn}
It makes sense to allow the mappings $\varphi \mapsto {^\varphi \!}\bfX$ and $\varphi \mapsto [a_\varphi, b_\varphi]$ to be multi-valued, so that the same chart can be used multiple times (e.g.\ if the trace $X$ goes back and forth between charts). The definition only depends on the smooth structure of $M$ (i.e.\ an equivalence class of atlases), since a rough path according to a particular atlas is uniquely extended to the maximal atlas for the smooth structure by reading \eqref{eq:compcond} as a definition of the right hand side. Also note that the definition is already fixed once we have ${^\varphi \!}\bfX|_{[a_\varphi,b_\varphi]}$ on a set of intervals, one for each chart, s.t.\ $\bigcup_\varphi (a_\varphi,b_\varphi) = (0,T)$: if the atlas if finite we call this a \emph{finite representation} of $\bfX$. The following example shows how our theory applies to the case of Stratonovich calculus of manifold-valued semimartingales.
\begin{expl}[Stratonovich rough path]\label{expl:strat}
	Let $X$ be an $M$-valued continuous semimartingale, i.e.\ $f(X)$ is a real-valued semimartingale for all $f \in C^\infty(M)$. Since semimartingales are a.s.\ bounded $p$-variation for any $p>2$ we only need to define a rough path above $X$ up to level $2$: we then define its \emph{Stratonovich rough path} in coordinates by
	\begin{equation}\label{eq:stratLift}
		\bfX^{\alpha\beta}_{st} \coloneqq \int_s^t X^\alpha_{su} \circ \dif X^\beta_u
	\end{equation}
	where the integral is intended in the Stratonovich sense. This is well-known to a.s.\ define a rough path in the linear setting, but since the above coordinate expression is taken accoding to a chart, we need the following lemma to establish that it defines a (stochastic) rough path in the sense of \autoref{def:rpM}.
\end{expl}
\begin{lem}
	Let $X$ be an $\bbR^m$-valued continuous semimartingale, $f \in C^\infty(\bbR^m,\bbR^n)$ and $\bfX$ be defined as in \eqref{eq:stratLift}. Then
	\begin{equation}
		(f_*\bfX)^{ij}_{st} = \int_s^t f^i(X) \circ \edif f^j(X).
	\end{equation}
	\begin{proof}
		Using that, in the linear setting, Stratonovich integrals a.s.\ coincide with rough integrals against the Stratonovich rough path \cite[Theorem 9.1]{FH14}
		\begin{align*}
			\int_s^t f^i(X) \circ \dif f^j(X) &= \int_s^t f^i(X) \partial_\gamma f^j(X) \circ \dif X\\
			&= \int_s^t \overline f^i \overline {\partial}_\gamma  \overline f^j(X) \dif  \bfX^\gamma \\
			&\approx f^i \partial_\gamma f^j(X_s)X_{st}^\gamma + (\partial_\alpha f^i \partial_\beta f^j + f^i \partial_{\alpha\beta}f^j)(X_s)\bfX_{st}^{\alpha\beta} \\
			&\approx f^i(X_s) f^j(X)_{st} + \partial_\alpha f^i \partial_\beta f^j(X_s)\bfX^{\alpha\beta}_{st}.
		\end{align*}
		We then have
		\[
		\int_s^t f^i(X)_{su} \dif f^j(X_u) = \int_s^t f^i(X) \dif f^j(X) - f^i(X_s) f^j(X)_{st} \approx \partial_\alpha f^i \partial_\beta f^j(X_s)\bfX^{\alpha\beta}_{st}
		\]
		concluding the proof.
	\end{proof}
\end{lem}
We can define a path valued in a fixed vector space $V$ and controlled by a manifold-valued rough path as follows:
\begin{defn}\label{def:controlledFixed}
	Let $\bfX \in \mathscr C^p_\omega([0,T],M)$. We define an $V$-valued \emph{$\bfX$-controlled path} $\bfH \in \mathscr D_{\bfX}(V)$ to be a collection ${^\varphi \!}\bfH \in \mathscr D_{{^\varphi\!}\bfX}(V)$ with $\varphi, a_\varphi, b_\varphi$ as in \autoref{def:rpM} and 
	\[
	{^\varphi \!}\bfH * \overline{(\varphi \circ \psi^{-1})}({^\psi\!}X) = {^\psi\!}\bfH.
	\]
\end{defn}
As for rough path, this only depends on the smooth structure of $M$. If we wish to define a controlled integrand, the trace of $H$ must lie in the bundle $\mathcal L(TM,W)$ for some vector space $W$, \say{above} $X$: for this we need a separate definition: 
\begin{defn}[Controlled integrand]\label{def:contrIntM}
	Let $\bfX \in \mathscr C^p_\omega([0,T],M)$. We define an $W$-valued \emph{$\bfX$-controlled integrand} $\bfH \in \mathscr D_{\bfX}(\mathcal L(TM, W))$ to be a collection ${^\varphi \!}\bfH \in \mathscr D_{{^\varphi\!}\bfX}(\mathcal L(\bbR^m,W))$ with $\varphi, a_\varphi, b_\varphi$ as in \autoref{def:rpM} and 
	\begin{equation}
		(\varphi \circ \psi^{-1})^* {^\varphi \!}\bfH = {^\psi\!}\bfH.
	\end{equation}
	The \emph{trace} of $\bfH$ is the path $H \coloneqq {^\varphi\!}H \circ T_X\varphi$, which is valued in the fibre of $X$ of the bundle $\mathcal L (T M, W)$ (if $W = \bbR^e$ then this bundle is $(T^*M)^e$.
\end{defn}
We can now define rough integration on manifolds.
\begin{defn}[Rough integral on manifolds]
	Let $\bfX \in \mathscr C^p_\omega([0,T],M)$ and $\bfH \in \mathscr D_{\bfX}(\mathcal L(TM, W))$. We define the \emph{rough integral}
	\begin{equation}\label{eq:roughIntMfd}
		\int_0^\cdot \bfH \dif \bfX \coloneqq \sum_{[s_\varphi,t_\varphi]} \int_{s_\varphi}^{t_\varphi} {^\varphi\!}\bfH \dif {^\varphi\!}\bfX \in \mathcal C^p_\omega([0,T],W)
	\end{equation}
	where we are summing over a finite partition of $[0,\cdot]$ whose intervals $[s_\varphi,t_\varphi]$ are indexed by charts $\varphi$ with the property that each $X_{[s_\varphi,t_\varphi]}$ is contained in the domain of $\varphi$. 
\end{defn}
This definition does not depend on the subdivision (since the integral taken w.r.t.\ two different subdivisions coincides with that take w.r.t.\ to their common refinement) and does not depend on the charts used, since, for another chart $\psi$ and $X_{[s,t]}$ contained in the domains of $\varphi$ and $\psi$
\begin{align*}
	\int_s^t {^\psi\!}\bfH \dif {^\psi\!}\bfX &= \int_s^t (\varphi \circ \psi^{-1})^* {^\varphi \!}\bfH \dif (\psi \circ \varphi^{-1})_* {^\varphi \!}\bfX \\
	&= \int_s^t (\psi \circ \varphi^{-1})^* (\varphi \circ \psi^{-1})^* {^\varphi \!}\bfH \dif {^\varphi \!}\bfX \\
	&= \int_s^t {^\varphi\!}\bfH \dif {^\varphi\!}\bfX
\end{align*}
by \autoref{thm:pushPull} and \autoref{prop:properties}.

We proceed to the topic of RDEs on manifolds: following the approach of \cite{E89}, and for maximum generality, we will consider both the driving rough path and the solution to be manifold-valued; for this purpose we let $N$ be an $n$-dimensional manifold.
\begin{defn}[RDEs on manifolds]
	Let $\bfX \in \mathscr C^p_\omega([0,T],M)$ and $F$ a section of the vector bundle $\mathcal L(TM,TN)$ (i.e.\ the vector bundle over $N \times M$ with fibres $\mathcal L(TM,TN)_{y,x} \coloneqq \mathcal L(T_xM,T_yN)$). We will say that $\bfY \in \mathscr C^p_\omega([0,T],M)$ is a solution to the RDE
	\begin{equation}\label{eq:RDEM}
		\dif \bfY = F(Y,X) \dif \bfX, \quad Y_0 = y_0
	\end{equation}
	if for all charts $\varphi$ and $\psi$ on $M$ and $N$ respectively we have $Y_0 = y_0$ and on all intervals $[s,t]$ s.t.\ $X_{[s,t]}$ is contained in the domain of $\varphi$ and $Y_{[s,t]}$ is contained in the domain of $\psi$
	\[
	\dif {^\psi\!}\bfY =  F({^\psi\!}Y,{^\varphi\!}X) \dif ({^\varphi\!}\bfX)\quad \text{on } [s,t]
	\]
	in the sense of \autoref{expl:nonAuto}, with initial condition ${^\psi\!}Y_s$.
\end{defn}
This does not depend on the chart: given other chats $\overline \varphi$, $\overline \psi$ we have
\begin{align*}
	\dif\begin{pmatrix} {^{\overline\varphi}\!}\bfX \\ {^{\overline\psi}\!}\bfY
	\end{pmatrix} = \dif\begin{pmatrix} (\overline\varphi \circ \varphi^{-1})_* {^{\varphi}\!}\bfX \\ (\overline\psi \circ \psi^{-1})_*{^{\psi}\!}\bfY
	\end{pmatrix} = \begin{pmatrix} \mathbbm 1 \\ F({^{\overline\psi}\!}Y,{^{\overline\varphi}\!}X)
	\end{pmatrix} \dif ({^{\overline\varphi}\!}\bfX )
\end{align*}
by \autoref{thm:changeVarRDEs} applied to the change of variable $(\overline\varphi \circ \varphi^{-1},
\overline\psi \circ \psi^{-1})$. Local existence can be inferred from \autoref{thm:localE}, which in particular implies global existence if $M$ is compact. More general conditions that guarantee global existence can be found in \cite{Wei18, Dri18}.

We end with a few brief remarks which link the topic of this paper to the existing literature, without elaborating on the details.
\begin{rem}[Stratonovich calculus]
	We can further expand on \autoref{expl:strat} to include the following: if $H$ is a $\mathcal L(TM,W)$-valued semimartingale above the $M$-valued semimartingale $X$ (i.e.\ $H_t \in \mathcal L(T_{X_t}M,W)$) the Stratonovich integral $\int H \dif X$ coincides a.s.\ with the rough integral $\int \bfH \dif \bfX$ against the Stratonovich rough path $\bfX$ for any choice of $\bfH$ with trace $H$, and the solution to the Stratonovich differential equation $\dif Y =  F(Y,X)\dif X$ coincides a.s.\ with the RDE \eqref{eq:RDEM}.
\end{rem}

\begin{rem}[The extrinsic viewpoint]\label{rem:extrinsic}
	In \cite{CDL15} the topic of manifold-valued theory of rough paths, rough integration (specifically of 1-forms) and RDEs was treated from the extrinsic point of view. Here $\bfX \in \mathscr C^p_\omega([0,T],\bbR^d)$ is defined in \cite[Definition 3.17]{CDL15} to be \emph{constrained} to a smoothly embedded manifold $M$ if its trace is $M$-valued and for all 1-forms $F \in \Gamma \mathcal L(T\bbR^d,W)$ ($\Gamma$ denoting the space of sections)
	\begin{equation}\label{eq:originalCon}
		\forall x \in M \ F(x)|_{T_xM} = 0 \ \Rightarrow \ \boldsymbol \int \overline F(X) \dif \bfX = 0 \in \mathscr C^p_\omega([0,T],W).
	\end{equation}
	In \cite[Corollary 3.32, Proposition 3.35]{CDL15} this is shown to be equivalent to the trace $X$ being $M$-valued and $(I \otimes Q(X_s)) \bfX^2_{st} \approx 0$, or equivalently to $(P(X_s) \otimes P(X_s)) \bfX^2_{st} \approx \bfX^2_{st}$, where for $x \in M$ $P(x)$ is the orthogonal projection $T_x\bbR^d \twoheadrightarrow T_xM$ and $Q \coloneqq \mathbbm 1 - P$. Moreover, one may replace $ \bfX^2_{st}$ with its antisymmetric part $(\wedge\bfX^2)_{st}$ in these identities (because $\odot \bfX^2$ is already fixed by the trace).
	
	This approach carries over to the case of higher $p$ considered here. If $\bfX \in \mathscr C^p_\omega([0,T],\bbR^d)$ we may say that it is \emph{constrained} to the smoothly embedded manifold $M$ if $\pi_*\bfX = \bfX$, where $\pi$ is the Riemannian projection of a tubular neighbourhood $U$ of $M$ onto $M$ (i.e.\ it maps a point in $U$ to the unique point on $M$ closest to it - this is well-defined and smooth on a thin enough tubular neighbourhood). This extends the definition of \cite{CDL15} since $D_x\pi = P(x)$. In order to generalise the equivalent condition $(I \otimes Q(X_s))(\wedge\bfX)_{st}$ we can take the log of our original condition, i.e.\ $\log \pi_*\bfX = \log \bfX$: this has the advantage of eliminating all the redundancies of the former (as explained in \cite[p.767]{LS06}), and its precise coordinate expression can be derived by using \cite[Definition 7.20]{FV10}, but at higher orders cannot be described in terms of antisymmetric tensors. The Chen-Strichartz formula \cite[Theorem 1.1]{Bau04}, however, expresses $\log \bfX$ as a Lie polynomial; the task of expressing it in a basis of the Lie algebra is more complex still \cite{Rei17}.
	
	Care must be taken when defining the rough integral of a controlled integrand, since it is no longer the case that for an $\bfX$-controlled integrand $\bfH$, $\int \bfH \dif \bfX$ (defined in the ordinary sense, where $\bfX$ is considered an element of $\mathscr C^p_\omega([0,T],\bbR^d)$) does not always only depend on the trace $H$ of $\bfX$ restricted to $TM$, although this is indeed the case when $\bfH$ is given by a 1-form: this is because if  $F \in \Gamma \mathcal L(T\bbR^d,W)$ ($\Gamma$ denoting the space of sections) vanishes on $TM$ (i.e.\ $F(x)P(x) = 0$ for $x \in M$) by \autoref{thm:pushPull} we have
	\begin{equation}\label{eq:constrInt}
		\begin{split}
			&\int \overline F(X) \dif \bfX
			= \int \overline F(X) \dif \pi_*\bfX =
			\int \pi^*\overline F(X) \dif \bfX = 0 \\
			\text{since}\quad &\pi^*\overline F(X) = (\overline F(X) * \overline \pi (X)) \cdot \overline{D\pi}(X) = \overline F(X) \cdot \overline P(X) = \overline{F(\cdot)P(\cdot)}(X) = 0
		\end{split}
	\end{equation}
	where we have used \autoref{prop:leibniz}. This then implies that if $F,G \in \Gamma \mathcal L(T\bbR^d,W)$ restrict to the same element of $\Gamma\mathcal L(TM, W)$ then $\int \overline F(X) \dif \bfX = \int \overline G(X) \dif \bfX$. Similarly, we have that if $\bfH, \bfK \in \mathcal D_{\bfX}(\mathcal L(TM,W))$ are such that $\pi^*\bfH = \pi^*\bfK$ then $\int \bfH \dif \bfX = \int \bfK \dif \bfX$, but this involves conditions on all levels of $\bfH$, not just the trace (for a simple counterexample where this identity fails when only assuming $H P(X) = K P (X)$ see \cite[Example 4.3]{ABCR}).
	
	Finally, the original definition of constrained rough path given by integration \eqref{eq:originalCon} also carries over to higher $p$. The fact that this is implied by $\pi_*\bfX = \bfX$ was shown in \eqref{eq:constrInt}. For the converse, we rewrite the identity as $(\mathbbm 1 - \pi)_* \bfX = 0$: the trace level is implied by the fact that $X$ is $M$-valued, and at orders $\geq 1$ the identity $\boldsymbol \int \overline Q(X) \dif \bfX = (\mathbbm 1 - \pi)_* \bfX$ is straightforward to check.
\end{rem}

\begin{rem}[Parallel transport and Cartan development]
	If $M$ has a connection $\nabla$ (with Christoffel symbols $\Gamma$), given $\bfX \in \mathscr C^p_\omega([0,T],M)$ we can define parallel transport of vectors above its trace $X$ as the solution to the RDE driven by $\bfX$, valued in $TM$, defined by the horizontal lift based at $X$: in local coordinates the parallel frame path $t \mapsto A_t \in T_{X_t}M$ satisfies
	\begin{equation}
		\dif A^\gamma = - \Gamma^\gamma_{\alpha\beta}(X)A^\beta \dif \bfX^\alpha.
	\end{equation}
	The Cartan development $\bfY$ of a $\boldsymbol Z \in \mathscr C^p_\omega([0,T],T_oM)$ (with $o \in M$ a fixed basepoint) can be viewed as the projection onto $M$ of an $FM$-valued RDE driven by $\boldsymbol Z$, defined the fundamental horizontal vector fields (the full solution additionally consists of a parallel frame above the developed path): in local coordinates
	\begin{equation}
		\begin{cases}
			\dif \bfY^k = A^k_\gamma \dif \boldsymbol Z^\gamma  \\
			\dif A^k_\gamma = - \Gamma^k_{ij}(Y) A^i_\alpha A^j_\gamma \dif \boldsymbol Z^\alpha
		\end{cases}.
	\end{equation}
	Here we have emphasised how development actually defines a (possibly explosive) $M$-valued rough path $\boldsymbol Y$, and not just its trace.
\end{rem}

\bibliographystyle{alpha} \addcontentsline{toc}{section}{References}
\bibliography{bibliography}

\end{document}